\documentclass[12pt,reqno]{article}

\usepackage[usenames]{color}
\usepackage{amssymb}
\usepackage{amsmath}
\usepackage{amsthm}
\usepackage{amsfonts}
\usepackage{amscd}
\usepackage{graphicx}

\usepackage{ifthen}

\usepackage[colorlinks=true,
linkcolor=webgreen,
filecolor=webbrown,
citecolor=webgreen]{hyperref}

\definecolor{webgreen}{rgb}{0,.5,0}
\definecolor{webbrown}{rgb}{.6,0,0}

\usepackage{color}
\usepackage{fullpage}
\usepackage{float}

\usepackage{graphics}
\usepackage{latexsym}
\usepackage{epsf}
\usepackage{breakurl}

\newcommand{\seqnum}[1]{\href{https://oeis.org/#1}{\rm \underline{#1}}}

\usepackage[utf8]{inputenc}
\usepackage[T1]{fontenc}
\usepackage{mathrsfs}
\usepackage{epsfig}
\usepackage{mathtools}
\usepackage{multicol}
\usepackage{enumitem}

\allowdisplaybreaks

\newcommand{\nnend}{\nonumber\\}
\DeclarePairedDelimiter{\ceil}{\lceil}{\rceil}
\DeclarePairedDelimiter{\na}{(}{)}

\newcommand{\C}{\ensuremath{\mathbb{C}}}
\newcommand{\Z}{\ensuremath{\mathbb{Z}}}
\newcommand{\N}{\ensuremath{\mathbb{N}}}

\begin{document}

\theoremstyle{plain}
\newtheorem{theorem}{Theorem}[section]
\newtheorem{conjecture}[theorem]{Conjecture}
\newtheorem{corollary}[theorem]{Corollary}
\newtheorem{lemma}[theorem]{Lemma}
\newtheorem{proposition}[theorem]{Proposition}

\theoremstyle{definition}
\newtheorem{example}[theorem]{Example}
\newtheorem{definition}[theorem]{Definition}
\newtheorem{notation}[theorem]{Notation}

\theoremstyle{remark}
\newtheorem{remark}[theorem]{Remark}

\numberwithin{equation}{section}

\begin{center}
\vskip 1cm{\LARGE\bf 
Completeness of Positive Linear Recurrence\\
\vskip .1in
Sequences
}
\vskip 1cm
\end{center}

\begin{multicols}{2}
\begin{center}
El\.zbieta Bo\l dyriew\\
Department of Mathematics\\
Colgate University\\
Hamilton, NY 13346\\ 
USA\\
\href{mailto:eboldyriew@colgate.edu}{\tt eboldyriew@colgate.edu} \\
\ \\

John Haviland\\
Department of Mathematics\\ 
University of Michigan\\
Ann Arbor, MI 48109 \\
USA\\
\href{mailto:havijw@umich.edu}{\tt havijw@umich.edu} \\
\ \\

Ph\'uc L\^am\\
Department of Mathematics\\
University of Rochester\\
Rochester, NY 14627\\ 
USA\\
\href{mailto:plam6@u.rochester.edu}{\tt plam6@u.rochester.edu}\\
\end{center}
\columnbreak
\begin{center}
John Lentfer\\
Department of Mathematics\\
Harvey Mudd College\\
Claremont, CA 91711\\
USA\\
\href{mailto:jlentfer@hmc.edu}{\tt jlentfer@hmc.edu} \\
\ \\

Steven J. Miller\\
Department of Mathematics and Statistics\\
Williams College\\
Williamstown, MA 01267\\
USA\\ 
\href{mailto:sjm1@williams.edu}{\tt sjm1@williams.edu} \\
\ \\

Fernando Trejos Su\'arez\\
Department of Mathematics\\
Yale University\\
New Haven, CT 06520\\
USA\\
\href{mailto:fernando.trejos@yale.edu}{\tt fernando.trejos@yale.edu} \\
\ \\

\end{center}
\end{multicols}

\begin{abstract}
	A sequence of positive integers is complete if every positive integer is a sum of distinct terms. A positive linear recurrence sequence (PLRS) is a sequence defined by a homogeneous linear recurrence relation with nonnegative coefficients of the form $H_{n+1} = c_1 H_n + \cdots + c_L H_{n-L+1}$ and a particular set of initial conditions.
	
	We seek to classify various PLRS's by completeness. With results on how completeness is affected by modifying the recurrence coefficients of a PLRS, we completely characterize completeness of several families of PLRS's as well as conjecturing criteria for more general families. Our primary method is applying Brown's criterion, which says that an increasing sequence is complete if and only if the first term is $1$ and each subsequent term is bounded above by the sum of all previous terms plus $1$. A survey of these results can be found in the authors' previous paper \cite{BHLLMT}.
	
	Finally, we adopt previous analytic work on PLRS's to find a more efficient way to check completeness. Specifically, the characteristic polynomial of any PLRS has exactly one positive root; by bounding the size of this root, the majority of sequences may be classified as complete or incomplete. Additionally, we show there exists an indeterminate region where the principal root does not reveal any information on completeness. We have conjectured precise bounds for this region.
\end{abstract}

\tableofcontents

\section{Introduction}

The Fibonacci numbers are one of the most studied integer sequences. One of their many interesting properties is that they can be used to construct a unique decomposition for any positive integer. Zeckendorf proved that every positive integer can be written uniquely as a sum of non-consecutive elements of the Fibonacci sequence, when indexed with the initial conditions $f_1=1,\ f_2=2$ and the recurrence $f_{n+1}=f_n+f_{n-1}$. Note that this is a just a shift of the indexing by one from the common initial conditions $F_0=0,\ F_1=1$ \seqnum{A000045}. For an arbitrary positive integer, this unique decomposition into Fibonacci numbers is called its \emph{Zeckendorf decomposition} \cite{Ze}. The result on the uniqueness and existence of such decompositions has been generalized to a much larger class of linear recurrence relations; the following definitions are from Miller and Wang \cite{MW}.

\begin{definition}\label{defn:goodrecurrencereldef} We say a sequence $\left(H_n\right)_{n=1}^\infty$ of positive integers is a \emph{Positive Linear Recurrence Sequence (PLRS)} if the following properties hold.
\begin{enumerate}
\item \emph{Recurrence relation:} There are non-negative integers $L, c_1, \dots, c_L$\label{c_i} such that \begin{equation} H_{n+1} \ = \ c_1 H_n + \cdots + c_L H_{n+1-L},\end{equation}  with $L, c_1$ and $c_L$ positive.
\item \emph{Initial conditions:} $H_1 = 1$, and for $1 \leq n < L$ we have
\begin{equation} H_{n+1} \ =\
c_1 H_n + c_2 H_{n-1} + \cdots + c_n H_{1}+1.\end{equation}
\end{enumerate}
\end{definition}

\begin{definition}[Legal decompositions]
We call a decomposition $\sum_{i=1}^{m} {a_i H_{m+1-i}}$\label{a_i} of a positive integer $N$ (and the sequence $\left(a_i\right)_{i=1}^{m}$) \emph{legal} \label{legal} if $a_1>0$, the other $a_i \geq 0$, and one of the following two conditions holds.
\begin{enumerate}

\item We have $m<L$ and $a_i=c_i$ for $1\leq i\leq m$.

\item There exists $s\in\{1,\dots, L\}$ such that
\begin{equation}
a_1\ = \ c_1,\ a_2\ = \ c_2,\ \cdots,\ a_{s-1}\ = \ c_{s-1}\ {\rm{and}}\ a_s<c_s,
\end{equation}
$a_{s+1}, \dots, a_{s+\ell} \ = \  0$ for some $\ell \geq 0$,
and $\left(b_i\right)_{i=1}^{m-s-\ell}$ (with $b_i = a_{s+\ell+i}$) is legal or empty.

\end{enumerate}
\end{definition}

The following theorem is due to Grabner and Tichy \cite{GT}, and stated in this form in Miller and Wang \cite{MW}.

\begin{theorem}[Generalized Zeckendorf's Theorem for PLRS]\label{thm:genZeckendorf} Let $\left(H_n\right)_{n=1}^\infty$ be a \emph{positive linear recurrence sequence}. Then there is a unique legal decomposition for each positive integer $N\geq 0$.
\end{theorem}

Next, we introduce \emph{completeness}, as defined by Hoggatt and King \cite{HK}.

\begin{definition}
An arbitrary sequence of positive integers $(a_i)_{i=1}^\infty$ is \emph{complete} if and only if every positive integer $n$ can be represented in the form $n=\sum_{i=1}^\infty \varepsilon_i a_i$, where $\varepsilon \in \{0,1\}$. A sequence that fails to be complete is \emph{incomplete}.
\end{definition}

In other words, a sequence of positive integers is complete if and only if each positive integer can be written as a sum of unique terms of the sequence. 

\begin{example}
    The Fibonacci sequence is complete. This follows directly from Zeckendorf's theorem, which is stronger statement, as it states that every positive integer may be written as the sum of \emph{non-consecutive} Fibonacci numbers. Completeness does not require that the decompositions use non-consecutive terms. 
    
    Note that unlike Zeckendorf decompositions, complete decompositions are not necessarily unique. In the case of the Fibonacci sequence, while the Zeckendorf decomposition of $10$ as $10=f_2+f_5=8+2$ is unique, we may find multiple complete decompositions, as with $10=2+8=2+3+5$.
\end{example}

After seeing this example, it is natural to ask if Theorem~\ref{thm:genZeckendorf} implies that all PLRS's are complete. Previous work in numeration systems by Gewurz and Merola \cite{GM} has shown that specific classes of recurrences as defined by Fraenkel \cite{Fr} are complete under their greedy expression. However, we cannot generalize this result to all PLRS's. For legal decompositions, the decomposition rule can permit some sequence terms to be used multiple. This is not allowed for completeness decompositions, where each unique term from the sequence can be used at most once.

\begin{example}
    The PLRS $H_{n+1} = H_n + 3H_{n-1}$ has terms $\left(1, 2, 5, 11, \ldots\right)$ \seqnum{A006138}. The unique \emph{legal} decomposition for $9$ is $1\cdot 5 + 2\cdot 2$, where the term $2$ is used twice. However, no \emph{complete} decomposition for $9$ exists. Adding all terms from the sequence less than $9$ is $1 + 2 + 5 = 8$, and to include $11$ or any subsequent term surpasses $9$.
\end{example}

We also make use of the following criterion for completeness of a sequence, due to Brown \cite{Br}.

\begin{theorem}[Brown's Criterion]
	If $a_n$ is a nondecreasing sequence, then $a_n$ is complete if and only if $a_1 = 1$ and for all $n > 1$,
	\begin{equation}\label{eqn:BrownsCrit}
		a_{n + 1} \leq 1 + \sum_{i = 1}^{n} a_i.
	\end{equation}
\end{theorem}

An immediate corollary is the following sufficient, though not necessary, condition for completeness, which we call the doubling criterion. The proof is left to the appendix, as Corollary~\ref{cor:doublingCritApx}.

\begin{corollary}[Doubling Criterion]\label{cor:doublingCrit}
	If $a_n$ is a nondecreasing sequence such that $a_n \leq 2 a_{n - 1}$ for all $n \geq 2$, then $a_n$ is complete.
\end{corollary}

\begin{remark}\label{two}
By considering the special case when $a_n = 2a_{n-1}$, this immediately implies that the doubling sequence itself $\left(1,2,4,8,\ldots\right)$ \seqnum{A000079} is complete.
\end{remark}

In this paper, we characterize many types of PLRS by whether they are complete or not complete.

\begin{notation}
We use the notation $[c_1, \ldots, c_L]$ to represent the PLRS defined by the recurrence $H_{n+1} = c_1 H_n + \cdots + c_L H_{n+1-L}$ and initial conditions as given in Definition~\ref{defn:goodrecurrencereldef}. When the context is clear, we also use $[c_1, \ldots, c_L]$ to refer to the coefficients themselves.
\end{notation}

A simple case to consider is when all coefficients $c_i$ for the sequence $[c_1,\ldots,c_L]$ are positive. The following result, proved in Section~\ref{sec:modifying}, completely characterizes these sequences are either complete or incomplete. 

\begin{theorem}\label{basic}
    If $(H_n)$ is a PLRS generated by positive coefficients $[c_1, \ldots, c_n]$, then $(H_n)$ is complete if and only if the coefficients are $[\underbrace{1,\ldots,1}_L]$ or $[\underbrace{1,\ldots,1}_{L-1},2]$ for $L \geq 1$.
\end{theorem}

The situation becomes much more complicated when we consider all PLRS's, in particular those that have at least one $0$ as a coefficient. In order to be able to make progress on determining completeness of these PLRS's, we develop several tools. The following three theorems are results that allow certain modifications of the coefficients $[c_1, \ldots, c_L]$ that generate a PLRS that is known to be complete or incomplete, and preserve completeness or incompleteness. They are proved in Section~\ref{sec:modifying}.

\begin{theorem}\label{lem:incompAddCoeff}
    Consider sequences $\left( G_{n} \right) = [c_1,\dots, c_{L}]$ and $\left( H_{n} \right)= [c_1,,\dots, c_{L},c_{L+1}]$, where $c_{L+1}$ is any positive integer. If $\left( G_{n} \right)$ is incomplete, then $\left( H_{n} \right)$ is incomplete as well.
\end{theorem}

\begin{theorem}\label{decreaseLastCoe}
    Consider sequences $\left(G_n\right)=[c_1,\ldots, c_{L-1}, c_L]$ and $\left(H_n\right)=[c_1,\ldots, c_{L-1}, k_L]$, where $1 \leq k_L \leq c_L$. If $\left(G_n\right)$ is complete, then $\left(H_n\right)$ is also complete.
\end{theorem}

\begin{theorem}\label{Adding M Theorem}    
    Consider sequences $\left(G_n\right)=[c_1,\ldots, c_{L-1}, c_L]$ and $\left(H_n\right)=[c_1,\ldots, c_{L-1} + c_L]$. If $\left(G_n\right)$ is incomplete, then $\left(H_n\right)$ is also incomplete.
\end{theorem}

The next theorem is a result that classifies a family of PLRS's as complete or incomplete. It is proved in Section~\ref{sec:famiilies}.

\begin{theorem}\label{thm:1onekzero} The sequence generated by $[1,\underbrace{0,\ldots,0}_k,N]$ is complete if and only if $1 \leq N \leq \left\lceil(k+2)(k+3)/{4}\right\rceil$, where $\lceil \cdot \rceil$ is the ceiling function.
\end{theorem}

The sequence of upper bounds on $N$ in Theorem~\ref{thm:1onekzero} is $(2,3,5,8,11,14,18,\ldots)$, as $k$ increases, which is a shift of \seqnum{A054925}.

We have a partial extension of these theorems to when there are $g$ initial ones followed by $k$ zeroes in the collection of coefficients.

\begin{theorem}\label{thm:gbon}
Consider a PLRS generated by coefficients $[\underbrace{1, \dots, 1}_{g}, \underbrace{0,\ldots,0}_{k},N]$, with $g,k \geq 1$.
\begin{enumerate}[itemsep=2ex, leftmargin=2em]
    \item For $g \geq k +\lceil \log _2 k\rceil$, the sequence is complete if and only if $1 \leq N \leq 2^{k+1}-1$.
    \item For $k \leq g \leq k +\lceil \log _2 k\rceil$, the sequence is complete if and only if $1 \leq N \leq 2^{k+1} - \lceil k/{2^{g-k}} \rceil$.
    \end{enumerate}
\end{theorem}

Finally, in Section~\ref{roots}, we give some results and conjectures on completeness based on the principal roots of a PLRS.  We determine some criteria for completeness based on the size of the principal root and find that there is a certain indeterminate region where the principal root does not reveal any information.

\section{Modifying sequences}\label{sec:modifying}

A basic question to ask is how far we can tweak the coefficients used to generate a sequence, yet preserve its completeness. The modifying process turns out to be well-behaved and heavily dependent on the location of coefficients that are changed. Before we start looking into implementing any changes to our sequences, we first need to understand the maximal complete sequence.

\subsection{Maximal complete sequence}
We introduce the maximal complete sequence, which serves an important role. First, we look at all complete sequences with only positive coefficients, and show Theorem~\ref{basic}, which states that any such sequence can only have the coefficients $[1,\ldots, 1]$ or $[1,\ldots,1,2]$. 

\begin{proof}[Proof of Theorem~\ref{basic}.]
Assume that $\left(H_n\right)$ is complete. By the definition of a PLRS and by Brown's criterion, we have
\begin{equation}
c_1 H_{L - 1} + c_2 H_{L - 2} + \cdots + c_{L - 1} H_{1} + 1 = H_L \leq 1 + H_1 + H_2 + \cdots + H_{L - 1}.
\end{equation}
Since $c_i \geq 1$ for $1 \leq i \leq L$, this implies that $c_i = 1$ for $1 \leq i < L$.
By the definition of a PLRS,
\begin{equation}
H_{L+1} = c_1 H_L + c_2 H_{L - 1} + \cdots + c_L H_1 = H_L + H_{L - 1} + \cdots + H_2 + c_L H_1.
\end{equation}
Combining this with Brown's criterion gives
\begin{align}
H_{L+1} = H_L + H_{L - 1} + \cdots + c_L H_1 &\leq 1 + H_1 + H_2 + \cdots + H_{L - 1} \nnend
c_L H_1 &\leq 1 + H_1 = 2.
\end{align}
Hence $c_L\leq 2$, which completes the forward direction of the proof.

We know that if the coefficients are just $[2]$, then the sequence is complete by Remark~\ref{two}. So, now assume that $c_1 = \cdots = c_{L - 1} = 1$ and $1 \leq c_L \leq 2$. We argue by strong induction on $n$ that $H_n$ satisfies Brown's criterion. We can show this explicitly for $1 \leq n < L$. First, if $n = 1$, then $H_n = 1$, as desired. Next, if $1 \leq n < L$, then
\begin{equation}
H_{n + 1} = c_1 H_n + \cdots + c_n H_1 + 1 = H_n + \cdots + H_1 + 1,
\end{equation}
so these terms satisfy Brown's criterion. Now assume that for some $n \geq L$, for all $n' < n$,
\begin{equation}
H_{n' + 1} \leq H_{n'} + \cdots + H_1 + 1.
\end{equation}
It follows that
\begin{align}
H_{n + 2} &= H_{n + 1} + \cdots + H_{n + 2 - L} + c_L H_{n + 1 - L}\nnend
&\leq H_{n + 1} + \cdots + H_{n + 2 - L} + 2H_{n + 1 - L}\nnend
&\leq H_{n + 1} + \cdots + H_{n + 2 - L} + H_{n + 1 - L} + (H_{n - L} + \cdots + H_1 + 1),\label{eqn:IHnonZeroPLRS}
\end{align}
where the inductive hypothesis was applied to $H_{n + 1 - L}$ to obtain (\ref{eqn:IHnonZeroPLRS}). This completes the induction.
\end{proof}

Now that we have found some complete sequences, it turns out that the sequence generated by the coefficient $[2]$, i.e., $\left(2^{n - 1}\right)$, is the maximal complete sequence. 

\begin{lemma}\label{clm:largestCompleteGaps}
The complete sequence with largest span in summands is $\left(2^{n - 1}\right)$.
\end{lemma}

\begin{proof}
Suppose there exists a complete sequence $\left(H_n\right)$ with the largest span in summands. As a complete sequence must satisfy Brown's criterion, it suffices to take $H_{n + 1} = 1 + \sum^{n}_{i=1} {H_i}$. Hence,  
\begin{align}
H_{n + 1} = 1 + \sum_1^{n} H_{i} &= 1 + \sum_1^{n - 1} H_{i} + H_{n} = 2H_{n}.
\end{align}
By the intial conditions for a PLRS, $H_1=1$ and $H_2=2$. Thus, $H_n = 2 H_{n-1}=2^{n-1}$. 
\end{proof}

\begin{remark}
Thus $\left(H_k\right) = \left(2^{k - 1}\right)$ is an inclusive upper bound for any complete sequence. 
\end{remark}

As it turns out, this sequence can be generated by multiple collections of coefficients. 

\begin{corollary}
A PLRS with coefficients $[\underbrace{1,\ldots,1}_{L-1},2]$ generates the sequence $H_n= 2^{n-1}$.
\end{corollary}

\begin{proof}
Consider the sequence $\left(H_n\right)$ 
generated by $[\underbrace{1,\ldots,1}_{L-1},2]$. We proceed by induction on $L$. Note $H_1=1=2^{1-1}$ by the definition of the PLRS. Now, suppose $H_k=2^{k-1}$ for $k\in\{1,\dots,n\}$. For $n<L$, note
\begin{align}
H_{n+1}&=c_1H_n+c_2H_{n-1}+\dots+c_n H_1+1\nnend
&=H_n+H_{n-1}+\dots+H_1+1\nnend
&=2^{n-1}+2^{n-2}+\dots+1+1=2^n. 
\end{align}
Hence, the claim holds for all $n<L$. Now, for $n\geq L$, note \begin{align} 
H_{n+1}&=c_1H_n+c_2H_{n-1}+\dots+c_L H_{n+1-L}\nnend
&=H_n+H_{n-1}+\dots+2H_{n+1-L}\nnend
&=2^{n-1}+2^{n-2}+\dots+2^{n-L+1}+2\cdot 2^{n-L}=2^n.
\end{align}
Thus, by induction, the claim holds for all $n,L\in\N$.
\end{proof}

\subsection{Modifications of sequences with arbitrary coefficients}
Modifying coefficients in order to preserve completeness proves to be a balancing act. Sometimes increasing a coefficient causes an incomplete sequence to become complete, while other times, increasing a coefficient causes a complete sequence to become incomplete. For example, $[1,0,0,0,0,0,15]$ is incomplete; increasing the second coefficient to $1$, i.e., $[1,1,0,0,0,0,15]$ is complete. Further increasing it to $2$, i.e., $[1,2,0,0,0,0,15]$ is again incomplete. To study how such modifications preserve completeness or incompleteness, we add a new definition to our toolbox. 

\begin{definition}
    For a sequence $\left(H_n\right)$, we define its \emph{$n$\textsuperscript{th} Brown's gap}
    \begin{equation}
    B_{H, n} \coloneqq 1 + \sum_{i=1}^{n-1}H_i - H_n.
    \end{equation}
\end{definition}

Thus, from Brown's criterion, $\left(H_n\right)$ is complete if and only if $B_{H, n} \geq 0$ for all $n \in \N$.

Our next questions is: What happens if we append one more coefficient to $[c_1,\ldots,c_L]$? It turns out that if our sequence is already incomplete, appending any new coefficients will never make it complete. This is Theorem~\ref{lem:incompAddCoeff}, which using are ready to prove using Brown's gap. 

\begin{proof}[Proof of Theorem~\ref{lem:incompAddCoeff}.]
By Brown's criterion, it is clear that $\left( G_{n} \right)$ is incomplete if and only if there exists $n$ such that $B_{G,n}<0$. We claim that for all $m$, $B_{H,m}\leq B_{G,m}$. If true, our lemma is proven: suppose $B_{G,n}<0$ for some $n$, we would see $B_{H,n}\leq B_{G,n}<0$, implying $\left( H_{n} \right)$ is incomplete as well.

We proceed by induction. Clearly, $B_{H,k}=B_{G,k}$ for $1\leq k \leq L$. Further, for $k=L$, we see
\begin{equation}
	B_{G,L+1}-B_{H,L+1}= 1+\sum_{i=1}^{L}G_{i} - G_{L+1} - \left(1+\sum_{i=1}^{L}H_{i} - H_{L+1} \right) =H_{L+1}-G_{L+1}=1>0
.\end{equation} 
Now, let $m \geq 2$ be arbitrary, and suppose 
\begin{equation}\label{Bequ1}
B_{H,\; L+m-1}\leq B_{G,\; L+m-1}.
\end{equation}
We wish to show that $B_{H,\; L+m}\leq B_{G,\; L+m}$.  Note that 
\begin{equation}\label{1first}
B_{H,\; L+m}-B_{H,\; L+m-1}= 2H_{L+m-1} - H_{L+m}.
\end{equation}
Similarly, 
\begin{equation}\label{2first}
B_{G,\; L+m}-B_{G,\; L+m-1}= 2G_{L+m-1} - G_{L+m}.
\end{equation}

We use Lemma~\ref{SequencesDifferencesAppendix}, which states that for all $k \geq 2$, $H_{L+k}-G_{L+k}\geq 2(H_{L+k-1}-G_{L+k-1})$. Applying this to \eqref{1first} and \eqref{2first}, we see that $B_{H,\; L+m}-B_{H,\; L+m-1}\leq B_{G,\; L+m}-B_{G,\; L+m-1}$. Summing this inequality to both sides of inequality \eqref{Bequ1}, we arrive at $B_{H,L+m}\leq B_{G,L+m}$, as desired.
\end{proof}

Now, we turn our attention to the behavior when we decrease the last coefficient for any complete sequence. In Theorem~\ref{decreaseLastCoe}, we find that decreasing the last coefficient for any complete sequence preserves completeness.

\begin{proof}[Proof of Theorem~\ref{decreaseLastCoe}.]
 Given that $\left(G_n\right)$ is complete, suppose for the sake of contradiction that there exists an incomplete $\left(H_n\right)$. Thus, let $m$ be the least such that
 \begin{equation} \label{eq:incomplete}
    H_m>1+\sum^{m-1}_{i=1}H_i.
\end{equation}
Simultaneously, as $\left(G_n\right)$ is complete, by Brown's criterion, \begin{equation}
    G_m\leq1+\sum^{m-1}_{i=1}G_i.
\end{equation} 
First, suppose $m\leq L$. However, for all $n\leq L$, $G_n=H_n$, hence \begin{equation}
    H_m=G_m\leq 1+\sum^{m-1}_{i=1}G_i=1+\sum^{m-1}_{i=1}H_i,
\end{equation}
which contradicts (\ref{eq:incomplete}). Now, suppose $m>L$. Therefore, 
\begin{equation}
    G_m\leq 1+\sum^{m-1}_{i=1}G_i = 1+\sum^{L}_{i=1}G_i+\sum_{i=L+1}^{m-1}G_i= 1+\sum^{L}_{i=1}H_i+\sum_{i=L+1}^{m-1}G_i.
\end{equation}
This implies
\begin{equation}
    1+\sum^{L}_{i=1}H_i \geq G_m-\sum_{i=L+1}^{m-1}G_i.
\end{equation}
Now, we know that 
\begin{equation}
    H_m>1+\sum^{m-1}_{i=1}H_i=1+\sum^{L}_{i=1}H_i+\sum_{i=L+1}^{m-1}H_i\geq G_m-\sum_{i=L+1}^{m-1}G_i+\sum_{i=L+1}^{m-1}H_i,
\end{equation}
and thus
\begin{align}\label{eq:contr}
    H_m-\sum_{i=L+1}^{m-1}H_i&> G_m-\sum_{i=L+1}^{m-1}G_i.
\end{align}
We claim that the opposite of (\ref{eq:contr}) is true, arguing by induction on $m$. For $m=L+1$, we obtain $G_{L+1}\geq H_{L+1}$ as $k_L\leq c_L$. Now, assume that
\begin{equation}
    G_m-\sum_{i=L+1}^{m-1}G_i \geq H_m-\sum_{i=L+1}^{m-1}H_i
\end{equation} is true for a positive integer $m$. Using the inductive hypothesis, it then follows that
\begin{align}
    G_{m+1}-\sum_{i=L+1}^{m}G_i=G_{m+1}-\sum_{i=L+1}^{m-1}G_i-G_m&\geq G_{m+1}-2G_m+H_m-\sum_{i=L+1}^{m-1}H_i. 
\end{align}
Finally, we use Lemma~\ref{Lemma3.37Appendix}, proved in Appendix~\ref{ProofsOfLemmas2}, which states that for all $k\in\N$, $H_{L+k+1}-2H_{L+k}\leq G_{L+k+1}-2G_{L+k}$. Note  
\begin{equation}
    G_{m+1}-2G_m+H_m-\sum_{i=L+1}^{m-1}H_i \geq  H_{m+1}-2H_m+H_m-\sum_{i=L+1}^{m-1}H_i =  H_{m+1}-\sum_{i=L+1}^{m}H_i,
\end{equation}
which does contradict (\ref{eq:contr}) for all $m>L$. 
Therefore, for all $m\in\N$, we have contradicted \eqref{eq:incomplete}. Hence, $\left(H_n\right)$ must be complete as well.
\end{proof}

The result above is crucial in our characterization of \textit{families} of complete sequences in Section~\ref{sec:famiilies}; finding one complete sequence allows us to decrease the last coefficient to find more. Next, we prove two lemmas that together prove Theorem~\ref{Adding M Theorem}.

\begin{lemma}\label{IncompExtension}
	Let $\left( G_{n} \right)$ be the sequence defined by $[c_1,\ldots, c_{L}]$, and let $\left( H_{n} \right)$ be the sequence defined by $[c_1,\ldots, c_{L-1}+1,\; c_{L}-1]$. If $\left( G_{n} \right)$ is incomplete, then $\left( H_{n} \right)$ must be incomplete as well. 
\end{lemma}

\begin{proof}
We claim that for all $m$, $B_{H,m}\leq B_{G,m}$. This lemma is proven using similar reasoning as for Lemma~\ref{lem:incompAddCoeff}. We proceed by induction. Clearly, $B_{H,k}=B_{G,k}$ for $1\leq k \leq L-1$. Further, for $k=L$, we see
\begin{equation}
	B_{G,L}-B_{H,L}=1+\sum_{i=1}^{L-1}G_{i} - G_{L} - \left(1+\sum_{i=1}^{L-1}H_{i} - H_{L}  \right) = H_{L}-G_{L}=1>0.
\end{equation} 
Now, let $m \geq 0$ be arbitrary, and suppose 
\begin{equation}\label{Bequ2}
B_{H,\; L+m}\leq B_{G,\; L+m}.
\end{equation}
We wish to show that $B_{H,\; L+m+1}\leq B_{G,\; L+m+1}$. Note that 
\begin{equation}\label{1second}
B_{H,\; L+m+1}-B_{H,\; L+m}=2H_{L+m}-H_{L+m+1},
\end{equation}
and similarly, 
\begin{equation}\label{2second}
B_{G,\; L+m+1}-B_{G,\; L+m}=2G_{L+m}-G_{L+m+1}.
\end{equation}
We use Lemma~\ref{Add 1 Lemma Appendix}, which says that for all $k \geq 0$, $H_{L+k+1}-G_{L+k+1}\geq 2\left( H_{L+k}-G_{L+k} \right)$. Applying it to \eqref{1second} and \eqref{2second}, we see $B_{H,\; L+m+1}-B_{H,\; L+m}\leq B_{G,\; L+m+1}-B_{G,\; L+m}$. Summing this inequality to both sides of inequality \eqref{Bequ2}, we conclude that $B_{H,L+m+1}\leq B_{G,L+m+1}$, as desired.
\end{proof}

How many times can Lemma~\ref{IncompExtension} be applied? Enough times to get all the way up to $[c_1,\ldots,c_{L-1}+c_L-1,1]$, but no further, as the last coefficient must remain positive to stay a PLRS.

\begin{lemma}\label{Last Case Adding M Theorem}
	Let $\left( G_{n} \right)$ be the sequence defined by $[c_1,\ldots , c_{L-1},1]$, and let $\left( H_{n} \right)$ be the sequence defined by $[c_1,\ldots , c_{L-1}+1]$. If $\left( G_{n} \right)$ is incomplete, then $\left( H_{n} \right)$ must be incomplete as well. 
\end{lemma}
	
\begin{remark}
	 Despite the similarities, Lemma~\ref{Last Case Adding M Theorem} is not directly implied by Lemma~\ref{IncompExtension}; both are necessary for the proof Theorem~\ref{Adding M Theorem}. Applying Lemma~\ref{IncompExtension} $(c_L-1)$ times proves that if $[c_1,\ldots , c_{L-1},c_L]$ is incomplete, then $[c_1,\ldots , c_{L-1}+c_L-1,1]$ is incomplete; at this point, we cannot apply the lemma further while maintaining a positive final coefficient to meet the definition of a PLRS. Hence the case of Lemma~\ref{Last Case Adding M Theorem} must be dealt with separately, in order to arrive at the full result of Theorem~\ref{Adding M Theorem}.
\end{remark}

\begin{proof}
The proof is similar to that of Lemma~\ref{IncompExtension}. We aim to show that $B_{H,m}\leq B_{G,m}$ for all $m$. Clearly $B_{H,k}=B_{G,k}$ for $1\leq k \leq L$. Further, for $k=L+1$, we see
\begin{equation}
B_{G,L+1}-B_{H,L+1}=\sum_{i=1}^{L}G_{i}-G_{L+1}-\left( 1+\sum_{i=1}^{L-1}H_{L}-H_{L+1} \right) =H_{L+1}-G_{L+1}=c_1>0.
\end{equation}
Now, let $m\geq 0$ be arbitrary, and suppose 
\begin{equation}\label{last brown gap inequality}
B_{H,L+m}\leq B_{G,L+m}.
\end{equation}
We wish to show that $B_{H,L+m+1}\leq B_{G,L+m+1}$. Note that 
\begin{equation}\label{ultimobrowngap}
B_{H,L+m+1}-B_{H,L+m}=2H_{L+m}-H_{L+m+1},
\end{equation}
and similarly
\begin{equation}\label{ultimobrowngap 2}
B_{G,L+m+1}-B_{G,L+m}=2G_{L+m}-G_{L+m+1}.
\end{equation}
We use Lemma~\ref{Last Case Adding M Appendix}, which states that for all $k\geq 0$, $H_{L+k+1}-G_{L+k+1}\geq 2\left( H_{L+k}-G_{L+k} \right) $. Applying it to equations \eqref{ultimobrowngap} and \eqref{ultimobrowngap 2}, we see $B_{H,L+m+1}-B_{H,L+m}\leq B_{G,L+m+1}-B_{G,L+m}$. Summing this inequality to both sides of Inequality \eqref{last brown gap inequality}, we conclude that $B_{H,L+m+1}\leq B_{G,L+m+1}$, as desired.
\end{proof}
	
Using these lemmas, we can now prove Theorem~\ref{Adding M Theorem}.

\begin{proof}[Proof of Theorem~\ref{Adding M Theorem}.]
	We apply Lemma~\ref{IncompExtension} $c_L-1$ times to conclude that if $[c_1,\ldots , c_{L-1},c_L]$ is incomplete, then $[c_1,\ldots , c_{L-1}+c_L-1,1]$ is incomplete. Finally, applying Lemma~\ref{Last Case Adding M Theorem}, we achieve the desired result.
\end{proof}

\section{Families of sequences}\label{sec:famiilies}
Recall that Theorem~\ref{decreaseLastCoe} says that given a complete PLRS, decreasing the last coefficient preserves its completeness. This raises a natural question: given the first $L-1$ coefficients $c_1, c_2, \dots, c_{L-1}$, what is the maximal $N$ such that $[c_1, c_2, \dots, c_{L-1}, N]$ is complete? In this section we explore this question.

\subsection{Using 1's and 0's as initial coefficients}

We first prove Theorem~\ref{thm:1onekzero}, which is about sequences with $1$ and $0$'s as the first coefficients. This proof is followed by a conjecture on classifying sequences of a similar form,  and then another conjecture on how complete sequences of these forms can be modified to obtain additional complete sequences and some progress toward proving it.

\begin{proof}[Proof of Theorem~\ref{thm:1onekzero}.]
First assume that $\left(H_n\right)$ is complete. By the definition of a PLRS, we can easily generate the first $k+2$ terms of the sequence: $H_i = i$ for all $1 \leq i \leq k+2$. We then have for all $n > k+1$,
\begin{equation}\label{eqn:termsThroughK+4}
    H_{n+1} = H_n + NH_{n-k-1},
\end{equation}
which implies that
\begin{equation}\label{eq:recrelation}
    H_{k+4} = H_{k+3} + NH_2 = H_{k+3} + 2N.
\end{equation}
By Brown's criterion,
\begin{equation}
    H_{k+4} \leq H_{k+3} + H_{k+2} + \cdots + H_1 + 1.
\end{equation}
By \eqref{eq:recrelation},
\begin{equation}
    H_{k+3} + 2N \leq H_{k+3} + H_{k+2} + \cdots + H_1 + 1,
\end{equation}
and we obtain
\begin{align}
    2N &\leq H_{k+2} + H_{k+1} + \cdots + H_1 + 1 \nnend
    &= (k+2) + (k+1) + \cdots + 1 + 1 \nnend
    &= \frac{(k+2)(k+3)}{2} + 1,
\end{align}
and thus we find
\begin{equation}
    N \leq \frac{(k+2)(k+3)}{4} + \frac{1}{2}. \label{takefloor}
\end{equation}
Since $N$ is an integer and $\left\lfloor (k + 2)(k + 3)/{4} + 1/2 \right\rfloor = \left\lceil (k + 2)(k + 3)/{4} \right\rceil$, we may take the floor of the right hand side of equation \eqref{takefloor}, and then $N \leq \left\lceil (k + 2)(k + 3)/{4} \right\rceil$.

We now prove that if $N \leq \left\lceil (k + 2)(k + 3)/{4} \right\rceil$, then $\left(H_n\right)$ is complete. We first show that if $N = \left\lceil (k + 2)(k + 3)/{4} \right\rceil$, then $\left(H_n\right)$ is complete. Taking the recurrence relation $H_{n+1} = H_n + NH_{n-k-1}$, and applying Brown's criterion gives
\begin{equation}
    H_{n+1} =H_n + NH_{n-k-1}
    \leq H_n +(N-2)H_{n-k-1} + H_{n-k-1} + H_{n-k-2} + \dots + H_1 +1.
\end{equation}
By Lemma~\ref{lem:sharp1onekzero}, we can expand $(N-2)H_{n-k-1}$ and find that
\begin{equation}
    H_{n+1} \leq H_n + H_{n-1} +\dots + H_{n-k} + H_{n-k-1} + H_{n-k-2} + \dots + H_1 +1.
\end{equation}
Hence, by Brown's criterion, the sequence $\left(H_n\right)$ is complete. Lastly, by Theorem~\ref{decreaseLastCoe}, for all positive $N < \left\lceil (k + 2)(k + 3)/{4} \right\rceil$, the sequence is also complete.
\end{proof}

We conjecture a bound on the last coefficient of a similar family sequence of PLRS as follows. The necessary condition for $N$ can be easily proven, similar to Theorem \ref{thm:1onekzero}. 

\begin{conjecture}\label{2onekzero}
The sequence generated by $[1, 1, \underbrace{0, \dots, 0}_{k}, N]$ is complete if and only if $1 \leq N \leq \lfloor (f_{k + 6} - k - 5)/4 \rfloor$, where $f_n$ are the Fibonacci numbers with $f_1=1, f_2=2$ and $\lfloor \cdot \rfloor$ is the floor function.
\end{conjecture}

We want to find a more general result for $[\underbrace{1, \dots, 1}_{g}, \underbrace{0, \dots, 0}_{k}, N]$, as seen in Figure~\ref{fig:FamiliesOfOne}.

\begin{figure}
    \centering
    \includegraphics[width=0.9\textwidth]{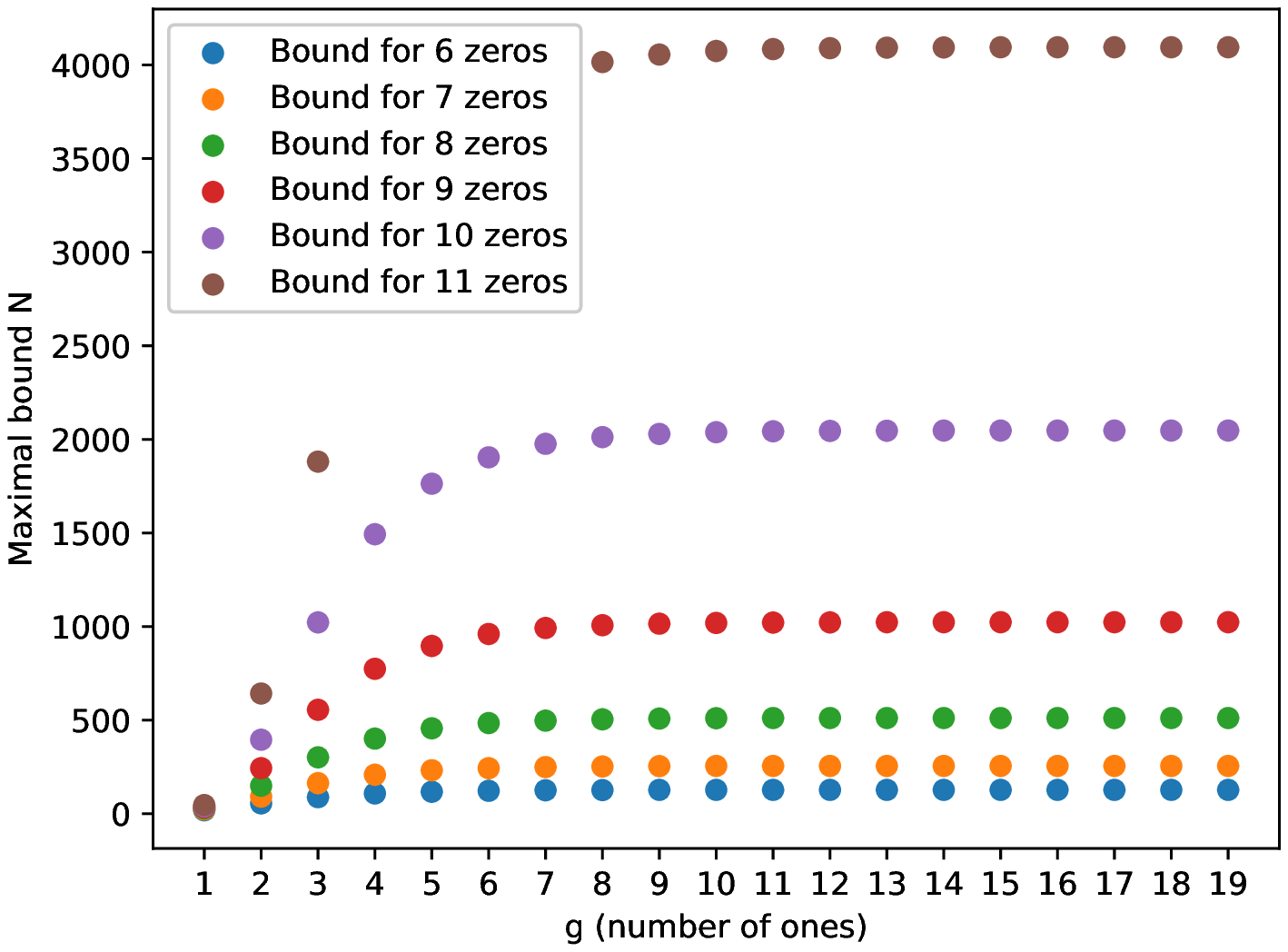}
    \caption{The maximal $N$ such that $[\protect\underbrace{1, \dots, 1}_{g}, \protect\underbrace{0, \dots, 0}_{k}, N]$ is complete, with $k$ and $g$ varying. Each color represents a fixed $k$.}
    \label{fig:FamiliesOfOne}
\end{figure}

Interestingly, we see that as we keep $k$ fixed and increase $g$, the bound increases, and then stays constant from some value of $g$ onward. This motivates the following conjecture.

\begin{conjecture}\label{addFrontOnes}
    If $[\underbrace{1, \dots, 1}_{g}, \underbrace{0, \dots, 0}_{k}, N]$ is complete, then so is $[\underbrace{1, \dots, 1}_{g+1}, \underbrace{0, \dots, 0}_{k}, N]$.
\end{conjecture}

We have made some progress towards this conjecture; in fact, we show the precise bound for $N$ for the case where $g \geq k$ in Theorem~\ref{thm:gbon}. 

\begin{theorem}\label{weak2Lcrit}
    The PLRS $(H_n)$ generated by $[c_1, c_2, \dots, c_L]$ is complete if
    \begin{equation}
    \begin{cases}
        B_{H, n} \geq 0, & \text{if $n < L$,}\\
        B_{H, n} > 0, & \text{if $L \leq n \leq 2L-1$.}
    \end{cases}
    \end{equation}
\end{theorem}

\begin{proof}
    Consider $L \geq 2$; we see that if $c_1 \geq 2$, then the sequence is automatically incomplete, so we need only consider $c_1 = 1$. For $B_n \coloneqq B_{H, n}$, and we show by induction on $n$ that $B_n > 0$ when $n \geq L$. 
    Suppose $B_n > 0$ for $L \leq n \leq m$ (with $m \geq 2L-1$). Then 
    \begin{align}
    B_{m+1} &= 1 + \sum_{i=1}^mH_i - H_{m+1}\nnend
    &= 1 + \sum_{i=1}^L H_i + \sum_{i=L+1}^m H_i - \left(H_m + \sum_{j=2}^L c_jH_{m+1-j}\right)\nnend
    &= 1 + \sum_{i=1}^L H_i + \sum_{i=L+1}^m \left(H_{i-1} + \sum_{j=2}^L c_jH_{i-j} \right) - \left(H_m + \sum_{j=2}^L c_jH_{m+1-j}\right)\nnend
    &= \left(1 + \sum_{i=1}^{m-1}H_i - H_m + H_L\right) + \sum_{j=2}^L c_j\left(\sum_{i=L+1}^m H_{i-j} - H_{m +1-j} \right)\nnend
    &= (B_m + H_L) + \sum_{j=2}^{L} c_j\left(1 + \sum_{i=j+1}^m H_{i-j} - H_{m +1-j} - 1 - \sum_{i = j+1}^L H_{i-j} \right) \nnend
    &= (B_m + H_L) + \sum_{j=2}^L c_j\left(B_{m+1-j} - 1 - \sum_{i = j+1}^L H_{i-j} \right) \nnend 
    &= B_m + \sum_{j=2}^L c_j(B_{m+1-j} - 1) + H_L - \sum_{i = 3}^L\sum_{j=2}^{i-1} c_jH_{i-j}\nnend 
    &= B_m + \sum_{j=2}^L c_j(B_{m+1-j} - 1) + H_L - \sum_{i=3}^L(H_i - H_{i-1} - 1)\nnend
    &= B_m + \sum_{j=2}^L c_j(B_{m+1-j} - 1) + (L-2) + H_L - \sum_{i=3}^L(H_i - H_{i-1})\nnend
    &= B_m + \sum_{j=2}^L c_j(B_{m+1-j} - 1) + L.
    \end{align}
    The last line is positive since $B_{m+1-j} - 1 \geq 0$ and $B_m, L > 0$. Our proof by induction is complete.
\end{proof}

\begin{lemma} \label{thm:sticky}
    The PLRS $(H_i)$ generated by $[\underbrace{1, \dots, 1}_{g}, \underbrace{0, \dots, 0}_{k}, 2^{k+1}]$ is incomplete if $g \geq k \geq 1$.
\end{lemma}

\begin{proof}
Suppose this sequence is complete. Note that 
\begin{equation}
    H_{2g+2}=H_{2g+1}+\dots+H_{g+2}+ 2^{k+1}H_{g+1-k}. 
\end{equation}
By applying Brown's criterion to $H_{2g+2}$, we see that
\begin{equation}
    2^{k+1}H_{g+1-k}\leq \sum_{i=1}^{g+1}H_i+1.
\end{equation}
Now, note $k$ is positive, so that $g+1-k\leq g+1$. Also, by the structure of the sequence, $H_i=2^{i-1}$ for $i\leq g+1$. Hence 
\begin{equation}
    2^{g+1} = 2^{k+1}H_{g+1-k} = 2^{k+1}2^{g-k} \leq \sum_{i=1}^{g+1}2^{i-1}+1 = 2^{g+1}.
\end{equation}
Therefore one may substitute previous inequalities with equalities and obtain 
\begin{equation} \label{2g+2 eq}
    H_{2g+2}=\sum^{2g+1}_{i=1}H_i+1.
\end{equation}
It follows immediately from (\ref{2g+2 eq}) that 
\begin{equation}
    \sum^{2g+2}_{i=1}H_i+1=2H_{2g+2}.
\end{equation}
Now, consider 
\begin{equation}
    H_{2g+3}=H_{2g+2}+H_{2g+1}+\dots+H_{g+3}+2^{k+1}H_{g+2-k}.
\end{equation} 
Since $g+2-k\leq g+1$ as $k\geq 1$, one gets \begin{equation}
    H_{g+2-k}M=2^{g+1-k}2^{k+1}=2^{g+2}=2(2^{g+1})=2H_{g+2}.
\end{equation}
Hence 
\begin{align}
    H_{2g+3}&=H_{2g+2}+H_{2g+1}+\dots+H_{g+3}+2H_{g+2} \nnend
    &=H_{2g+2}+\left(H_{2g+1}+\dots+H_{g+3}+H_{g+2}+H_{g+2}\right) \nnend
    &>H_{2g+2}+\left(H_{2g+1}+\dots+H_{g+3}+H_{g+2}+H_{g+1-k}\right) \nnend
    & =2H_{2g+2}=\sum^{2g+2}_{i=1}H_i+1.
\end{align}
So $H_{2g+3}$ causes Brown's criterion to fail, rendering whole sequence incomplete.
\end{proof}

We now show the stabilizing behavior of the bound mentioned above.

\begin{lemma}\label{thm:SharpgAndlog2k}
    If $g \geq k + \lceil \log_2k \rceil$, then $[\underbrace{1, \dots, 1}_{g}, \underbrace{0, \dots, 0}_{k}, 2^{k+1} - 1]$ is complete.
\end{lemma}

\begin{proof}
Define $\left( f_n \right)=[\underbrace{1,\ldots , 1}_{g}]$, and $\left( H_n \right)=[\underbrace{1, \dots, 1}_{g}, \underbrace{0, \dots, 0}_{k}, 2^{k+1} - 1]$. We can calculate the terms of $\left( f_n \right)$ and $\left( H_n \right)$ up to $2g+1$. Namely,
\begin{alignat}{2}
H_{n} &= f_{n}=2^{n-1}, &&\qquad \text{if $1\leq n\leq g$;}  \nnend
H_{g+n} &= f_{g+n}+2^{n-1}, &&\qquad \text{if $1\leq n\leq k+1$;}\nnend
H_{g+k+1+n} &= f_{g+k+1+n}+\left( 2^{k+1}-1 \right) \left( 2^{n}+2^{n-2}\left( n-1 \right)  \right), &&\qquad \text{if $1\leq n\leq g-k$;}\nnend
f_{g+n} &= 2^{g+n-1}-2^{n-2}\left( n+1 \right), &&\qquad \text{if $1 \leq n \leq g$.}
\end{alignat}
The third and fourth lines are verified in Lemmas~\ref{Line3Terms} and \ref{Line4Terms}, respectively. We show that the conditions in Theorem~\ref{weak2Lcrit} hold for $\left(H_n\right)$. We can verify directly that Brown's criterion holds for the first $(2g+1)$ terms of $\left( H_n \right)$; in fact, for $B_n \coloneqq B_{H, n}$, we get 
\begin{equation}
\begin{cases}
    B_n \geq 0, & \text{if $1\leq n\leq g + k$;}\\
    B_n > 0, & \text{if $g+k+1 \leq n \leq 2g+1$.}
\end{cases}
\end{equation}
Thus, it remains to show that $B_n > 0$ for $2g+2\leq n \leq 2\left( g+k \right) -1$. 
\begin{enumerate}[label={Case \arabic*:}, leftmargin=*]
\item $2g+2 \leq n \leq 2g+k+1$.

Define $b(n) \coloneqq H_{n}-f_{n}$. Note that $b(n)\geq 0$, and by induction, $b(n) > 0$ for all $n \geq g+1$. For $n\geq g+k+2$,
\begin{align}
f_{n}+b(n) &=H_{n}\nnend
&= H_{n-1}+H_{n-2}+\cdots +H_{n-g}+\left( 2^{k+1}-1 \right) H_{n-(g+k+1)} \nnend
&= \sum_{i=1}^{g}f_{n-i}+\sum_{i=1}^{g}b\left( n-i \right) +\left( 2^{k+1}-1 \right) H_{n-(g+k+1)}.
\end{align}
Since $f_{n}=\sum_{i=1}^{g}f_{n-i}$, 
\begin{equation}
b(n)=\sum_{i=1}^{g}b\left( n-i \right) +\left( 2^{k+1}-1 \right) H_{n-(g+k+1)}.
\end{equation}
Thus, for any $n\geq 2g+2$, 
\begin{align}\label{LastTerm}
B_n &= 1+\sum_{i=1}^{n-1}H_{i}-H_{n} \nnend
&= 1+\sum_{i=1}^{n-1}\left( f_{i}+b(i) \right) -\left( f_{n}+b(n) \right) \nnend
&= \left( 1+\sum_{i=1}^{n-1}f_{i}-f_{n} \right) -\left( 2^{k+1}-1 \right) H_{n-(g+k+1)}+\sum_{i=g+1}^{n-\left( g+1 \right) }b(i)\nnend
& > \left( 1+\sum_{i=1}^{n-1}f_{i}-f_{n} \right) -\left( 2^{k+1}-1 \right) H_{n-(g+k+1)}.
\end{align}
We are to show that the last term is nonnegative. As $n-\left( g+k+1 \right) \leq g$, 
\begin{align}
& 1+\sum_{i=1}^{n-1}f_{i}-f_{n}-\left( 2^{k+1}-1 \right) H_{n-\left( g+k+1 \right) }\nnend
&= 1+\sum_{i=1}^{n-\left( g+1 \right) }f_{i}-\left( 2^{k+1}-1 \right) H_{n-(g+k+1)}\nnend
&= 1+\sum_{i=1}^{g}f_{i}+\sum_{i=1}^{n-\left( 2g+1 \right) }f_{g+i}-\left( 2^{k+1}-1 \right) \cdot 2^{n-\left( g+k+1 \right) -1}\nnend
&= 1+\sum_{i=1}^{g}2^{i-1}+\sum_{i=1}^{n-\left( 2g+1 \right) }\left( 2^{g+i-1}-2^{i-2}\left( i+1 \right)  \right) -2^{n-g-1}+2^{n-\left( g+k+1 \right) -1}\nnend
&= 2^{n-\left( g+k+1 \right) -1}-\sum_{i=1}^{n-\left( 2g+1 \right) }2^{i-2}\left( i-1 \right) -\sum_{i=1}^{n-\left( 2g+1 \right) }2^{i-1}\nnend
&= 2^{n-\left( g+k+1 \right) -1}-\left( 2^{n-\left( 2g+2 \right) }\left( n-\left( 2g+3 \right) \right) +1 \right) -\left( 2^{n-\left( 2g+2 \right) }-1 \right)\nnend
&= 2^{n-\left( g+k+1 \right) -1}-2^{n-\left( 2g+2 \right) }\left( n-\left( 2g+2 \right)  \right)\nnend
&= 2^{n-\left( 2g+2 \right) }\left( 2^{g-k}-\left( n-\left( 2g+2 \right)  \right)  \right)\nnend
& \geq 2^{n-\left( 2g+2 \right) }\left( 2^{g-k}-\left( k-1 \right)  \right)\nnend
& > 0.
\end{align}
Note that the last line comes from $g \geq k+\log_2k$, which implies $2^{g-k }\geq k > k-1$.

\item $2g+k+2 \leq n \leq 2g+2k+1$.

We show that $B_{n+1} \geq B_n$ for $2g+k+2 \leq n < 2g+2k+1$, and that $B_{2g+k+2} > 0$. 
\begin{align}\label{differenceBGap}
    B_{n+1} - B_n &= 2H_n - H_{n+1} \nnend
    &= 2H_n - \left(\sum_{i=n-g+1}^n H_i +(2^{k+1} -1)H_{n - (g+k)} \right) \nnend
    &= \left(H_n - \sum_{i=n-g+1}^n H_i \right) +(2^{k+1} -1)H_{n - (g+k)} \nnend
    &= H_{n-g} - (2^{k+1} -1)(H_{n - (g+k)} - H_{n-(g+k+1)}). 
    \intertext{Replace $n$ by $2g+k+1+m$ with $1 \leq m \leq k$ to obtain}
    &= H_{(g+k+1)+m} - (2^{k+1}-1)(H_{g+m+1} - H_{g+m}) \nnend
    &= H_{(g+k+1)+m} - (2^{k+1}-1)(2^{g+m-1} - 2^{m-2}(m+1)).
\end{align}
For $1 \leq m \leq g-k$, we have an explicit formula for $H_{(g+k+1)+m}$, so we can substitute directly to show that \eqref{differenceBGap} is nonnegative. Thus, if $g-k \geq k$ (i.e., $g \geq 2k$), then this holds for all $1 \leq m \leq k$. If $g - k < k$ (i.e., $g < 2k$), then from Lemma~\ref{lem:Gap2}, \eqref{differenceBGap} is nonnegative. Thus, $B_{n+1} \geq B_n$ for all $2g+k+2 \leq n \leq 2g+k+1$. It remains to show that $B_{2g+k+2} > 0$, which we can do by directly substituting the explicit formulas.\qedhere
\end{enumerate}
\end{proof}

Combining these lemmas, we can prove the first part of Theorem~\ref{thm:gbon}.

\begin{proof}[Proof of Theorem~\ref{thm:gbon}.1.]
From Lemmas~\ref{thm:sticky} and~\ref{thm:SharpgAndlog2k}, the bound for $N$ is precisely $2^{k+1} - 1$ when $g \geq k + \lceil \log_2k \rceil$.
\end{proof}

Next, we consider when $k\leq g \leq k+\ceil{\log_2k}$, and prove the second part of Theorem~\ref{thm:gbon} using similar methods.

\begin{proof}[Proof of Theorem~\ref{thm:gbon}.2.]
First, we show that 
for $N>2^{k+1}-\ceil{k/{2^{g-k}}}$, $\left(H_i\right)$ is incomplete, and suppose $k \geq 2$. Let us calculate the initial $L = g+k+1$ terms of the sequence. Note
\begin{alignat}{2}
    H_n&=2^{n-1} &&\qquad\text{for all } 1\leq n \leq g+1 \nnend
    H_{g+n}&=2^{g+n-1}-2^{n-2}(n-1) &&\qquad\text{for all } 1\leq n \leq k+1.
\end{alignat}
Let $B_i \coloneqq B_{H,i}$. Then, we consider Brown's gap $B_{2g+k+2}$,
\begin{align}
B_{2g+k+2} &= \na*{1+\sum_{i = 1}^{2g+k+1}{H_i}} - H_{2g+k+2} \nnend
&= \na*{1+\sum_{i = 1}^{2g+k+1}{H_i}} - \na*{\sum_{i = g+k+2}^{2g+k+1}{H_i}+NH_{g+1}} \nnend
&=\na*{1+\sum_{i = 1}^{g+k+1}H_i} - NH_{g+1} \nnend
&=1+\sum_{i = 1}^{g}H_i+\sum_{i = g+1}^{g+k+1}H_i - NH_{g+1}\nnend
&=1+\sum_{i = 1}^{g}2^{i-1}+\sum_{i = 1}^{k+1}\na*{2^{g+i-1}-2^{i-2}\na*{i-1}} - 2^g N\nnend
&=2^{g+k+1}-\sum_{i = 1}^{k}2^{i-1}i - 2^g N\nnend
 &=2^{g+k+1}-2^k(k-1)-1 - 2^gN .
\intertext{Now, $N>2^{k+1}-\ceil*{k/{2^{g-k}}}$ by assumption so it follows that $N\geq 2^{k+1}-k/{2^{g-k}}+1$, hence}
&\leq 2^{g+k+1}-2^k(k-1)-1 - 2^{g}\na*{2^{k+1}-\frac{k}{2^{g-k}}+1}\nnend 
&=2^{k}-2^g-1,
\end{align}
which must be negative as $g\geq k$. So $\left(H_n\right)$ fails Brown's criterion at the $(2g+k+1)$st term, rendering the sequence incomplete.

Now we can show that for $N=2^{k+1}-\ceil{k/{2^{g-k}}}$, $\left(H_i\right)$ is complete by Theorem~\ref{weak2Lcrit}. We can easily verify that $B_n\geq 0$ for all $1\leq n \leq g+k+1$ and $B_{g+k+1} > 0$; it remains to show that $B_n > 0$ for $g+k+2 \leq n \leq 2g+2k+1$. We consider two cases.  
\begin{enumerate}[label={Case \arabic*:}, leftmargin=*]
    \item $2 \leq n-(g+k)\leq g+1$.
        
        We want to show that $B_{n+1}\geq B_n$ for all $2\leq n - (g+k) \leq g+1$ and that $B_{g+k+2} > 0$. Now,
    \begin{align}
        B_n&=1+\sum_{i=1}^{n-1}{H_i}-H_n \nnend
        &=1+\sum_{i=1}^{n-1}{H_i}-\na*{\sum_{i=n-g}^{n-1}{H_i}+NH_{n-(g+k+1)}} \nnend
        &=1+\sum_{i=1}^{n-g-1}{H_i}-NH_{n-(g+k+1)}. 
    \end{align} 
Then, note that
\begin{align}
    B_{n+1}-B_n&=H_{n-g}-N\na*{H_{n-(g+k)}-H_{n-(g+k+1)}} \nnend
    &=H_{n-g}-N\na*{2^{n-(g+k+1)}-2^{n-(g+k+2)}},
    \intertext{and by assumption,}
    &=H_{n-g}-\na*{2^{k+1}-\ceil[\Big]{{\frac{k}{2^{g-k}}}}}{2^{n-(g+k+2)}} \nnend
    &={2^{n-(g+k+2)}}\ceil[\Big]{{\frac{k}{2^{g-k}}}}-\na*{2^{n-g-1}-H_{n-g}}.
\end{align}
If $n-g\leq g+1$, then $2^{n-g-1}-H_{n-g}=0$, so $B_{n+1}-B_n>0$. If $g+2\leq n-g \leq g+k+1$, then 
\begin{equation}
    2^{n-g-1}-H_{n-g}=2^{n-2g-2}\na*{n-2g-1}\leq 2^{n-(g+k+2)}\frac{k}{2^{g-k}}\leq 2^{n-(g+k+2)}\ceil[\Big]{{\frac{k}{2^{g-k}}}},
\end{equation}
so that $B_{n+1}-B_n \geq 0$. In any case, $B_{n+1}\geq B_n$. We can verify directly that $B_{g+k+2}>0$, completing this case.

\item $g\leq n-(g+k)\leq g+k+1$.

From the previous case, $B_{2g+k+2}\geq B_{2g+k+1}>0$. Now,
\begin{align}
B_n&=1+\sum_{i=1}^{n-g-1}H_i - NH_{n-(g+k+1)} \nnend
&=1+\sum_{i=1}^{n-2g-1}H_i+\sum_{i=n-2g}^{n-g-1}H_i- NH_{n-(g+k+1)} \nnend
&=1+\sum_{i=1}^{n-2g-1}H_i+H_{n-g}- NH_{n-(2g+k+1)}- NH_{n-(g+k+1)}.
\intertext{Substituting $n=2g+k+1+m$ for $1\leq m \leq k$,}
&=1+\sum_{i=1}^{k+m}H_i+H_{g+k+1+m}-N\na*{H_m+H_{g+m}} \nnend
&\geq H_{k+m+1}+H_{g+k+1+m}-N\na*{2^{m-1}+2^{g+m-1}-2^{m-2}(m-1)}. \label{Cm}
\end{align}
Let $C_m \coloneqq H_{k+m+1} + H_{g+k+1+m}- N \na*{2^{m-1}+2^{g+m-1}-2^{m-2}(m-1)}$, from equation \eqref{Cm}. We show by strong induction that $C_m > 0$. By direct computation, $C_1 > 0$. Suppose it holds for all values from $1$ to $m-1$ for $m\geq 2$. Then by the induction hypothesis, 
\begin{align}
    H_{g+k+1+m}&=\na*{H_{g+k+m}+\cdots+H_{g+k+2}}+\na*{H_{g+k+1}+\cdots+H_{m+k+1}}+NH_m  \nnend 
    & > \sum_{i=1}^{m-1}\na*{N\na*{2^{i-1}+2^{g+1-i}-2^{i-2}(i-1)}-H_{k+i+1}}+\nnend &\hspace{25mm}+\na*{2^{g+k}+\cdots+2^{m+k}-\sum_{i=1}^{k+1}2^{i-2}(i-1)}+2^{m-1}N \nnend
    &=N\na*{2^m-1+2^{g+m+1}-2^g-2^{m-2}(m-3)-1} - \nnend
    &\hspace{25mm} -\sum_{i=k+2}^{k+m}H_i+\na*{2^{g+k+1}-2^{m+k}-2^k(k-1)-1} \nnend
    &\geq N\na*{2^{m-1}+2^{g+m-1}-2^{m-2}(m-1)}-\na*{2^g+2-2^m}N - \nnend 
    &\hspace{25mm} - \sum_{i=k+m-g}^{k+m} H_i + \na*{2^{g+k+1}-2^{m+k}-2^k(k-1)-1},
\end{align}
where $H_i = 0$ for nonpositive $i$. Hence, 
\begin{align}
   C_m&=H_{g+k+1+m}- N \na*{2^{m-1}+2^{g+m-1}-2^{m-2}(m-1)}+H_{k+m+1}\nnend
   & > \na*{H_{k+m+1}-\sum_{i=k+m-g}^{k+m}H_i}+\na*{2^{g+k+1}-2^{m+k}-2^k(k-1)-1} - \nnend
   &\hspace{80mm} -\na*{2^g+2-2^m}N \nnend 
   &=1+\na*{2^{g+k+1}-2^{m+k}-2^k(k-1)-1}-\na*{2^g+2-2^m}\left(2^{k+1}-\left\lceil\frac{k}{2^{g-k}}\right\rceil\right) \nnend
   &=2^{m+k}-2^k\na*{k+3}+\na*{2^g+2-2^m}\left\lceil\frac{k}{2^{g-k}}\right\rceil \nnend
   &\geq 2^{m+k}-2^k\left(k+3\right)+\left(2^g+2-2^m\right) \frac{k}{2^{g-k}}\nnend 
   &=2^{m+k}-3\cdot 2^k - \left(2^m-2\right)\frac{k}{2^{g-k}}\nnend
   &=\na*{2^m-3}\left(2^k-\frac{k}{2^{g-k}}\right)-\frac{k}{2^{g-k}} \nnend
   &\geq 2^k-\frac{2k}{2^{g-k}} \geq 2^k-2k \geq 0.
\end{align}
 This completes the induction, so $B_n\geq C_m > 0$.
\end{enumerate}
Since both cases are satisfied, $\left(H_i\right)$ is complete. 
\end{proof}

\begin{remark}
The case $k = 1$ is characterized in Lemma~\ref{lem:failAt2L-1}.
\end{remark}

\subsection{The ``\texorpdfstring{$2L - 1$}{2L - 1} conjecture''}

We conjecture a strengthened version of Theorem~\ref{weak2Lcrit} as follows.

\begin{conjecture}\label{2Lcrit}
    The PLRS $\left(H_n\right)$ defined by $[c_1, \dots, c_L]$ is complete if $B_{H, n} \geq 0$ for all $n \leq 2L - 1$, i.e., Brown's criterion holds for the first $2L-1$ terms.
\end{conjecture}

When using Brown's criterion, it would be very helpful to know how many terms must be checked to be sure that a PLRS is complete. This conjecture, if true, would be a powerful tool to do so. We do not know yet if such a threshold exists for each $L$; however, if it does, then it is at least $2L-1$, as shown by the following example, where $k+2=L$.

\begin{lemma}\label{lem:failAt2L-1}
   The sequence $[1,\dots,1,0,4]$, with $k$ ones, where $k \geq 1$, is always incomplete. Moreover, it first fails Brown's criterion on the $(2k+3)$\textsuperscript{rd} term. 
\end{lemma}

\begin{proof}
We have the recurrence relation $H_{n+1} = H_n + \dots + H_{n-k+1} + 4H_{n-k-1}$. We show that the term in the $(2k+3)$\textsuperscript{rd} position in the sequence fails Brown's criterion. First, 
\begin{equation}
    H_{2k+3} = H_{2k+2} + \dots + H_{k+3} + 4H_{k+1}. 
\end{equation}
Next, we observe that for $1 \leq j \leq k+1$, we have  $H_j=2^{j-1}$. Additionally, $H_{k+2} = 2^{k+1}-1$. Thus,
\begin{equation}
    2H_{k+1} = 2^{k+1}>2^{k+1}-1 = H_{k+2}.
\end{equation}
We also note that $H_{k+1} = H_k +\dots + H_1 +1$. Putting everything together,
\begin{align}
    H_{2k+3} &= H_{2k+2} + \dots + H_{k+3} + 4H_{k+1}\nnend
    &= H_{2k+2} + \dots + H_{k+3} + 3H_{k+1} + H_k + \dots + H_1 +1 \nnend
    &> H_{2k+2} + \dots + H_{k+3} + H_{k+2} + H_{k+1} + H_k + \dots + H_1 +1.
\end{align}
Hence, we have shown that $[1,\dots,1,0,4]$, with $k\geq 1$ ones, is incomplete, as it fails Brown's criterion on the $(2k+3)$\textsuperscript{rd} term.

We now show that Brown's criterion holds for the first $(2k+2)$ terms. For $1 \leq j \leq k+1$, we have $H_j=2^{j-1}$, which satisfies the equality $H_{j+1}=H_j +\dots + H_1 + 1$. When $k+2 \leq j \leq 2k+2$, 
\begin{equation}
    H_{j+1} = H_j + \dots + H_{j-k+1} + 4H_{j-k-1}.
\end{equation}
Note that $H_{j-k-1}=H_{j-k+2}+\dots+H_1+1$ as $1 \leq j-k-1 \leq k + 1$, so
\begin{equation}
    H_{j+1} = H_j + \dots + H_{j-k+1} + 2H_{j-k-1} + H_{j-k-1} + H_{j-k-2} + \dots + H_1 +1
\end{equation}
and as $2H_{j-k-1} = 2^{j-k-1}=H_{j-k}$,  we see
\begin{equation}
    H_{j+1} = H_j + \dots + H_{j-k+1} + H_{j-k} + H_{j-k-1} + H_{j-k-1} + H_{j-k-2} + \dots + H_1 +1.
\end{equation}
Hence, this equality satisfies Brown's criterion for terms $k+2 \leq j \leq 2k+2$.
\end{proof}

Assuming this conjecture, we can explore sequences of the form $[1, 0, \dots, 0, 1, \dots, 1, N]$ further. In Theorems~\ref{sum2L-2m} and \ref{sum2L-2m 2nd}, we show that the bound on $N$ for $[1, \underbrace{0, \dots, 0}_{L-m-2}, \underbrace{1, \dots, 1}_{m}, N]$ strictly increases if we keep $L$ fixed and increase $m$ from $0$ to $L-3$, i.e., switching the coefficients from $0$ to $1$ gradually from the end so that at least one $0$ remains. We first state a following powerful lemma that is contingent on this conjecture.

\begin{lemma}[Conditional]\label{firstL+2}
	Let $\left( H_{n} \right)$ defined by $[1,0,\dots, 0,1,\dots, 1,N]$ be a sequence with $L$ coefficients, $m$ of which are ones. Then, if $(H_n)$ is incomplete, it must fail Brown's criterion at the $(L+1)$st or $(L+2)$nd term. In other words, if $H_{L+1}\leq 1+\sum_{i=1}^{L}H_{i}$ and $H_{L+2}\leq 1+\sum_{i=1}^{L+1}H_{i}$, then $\left( H_{n} \right)$ is complete.
\end{lemma}

The proof of this lemma is deferred to Lemma~\ref{firstL+2append} of Appendix~\ref{apx:sec3lemmas}.

\begin{theorem}\label{sum2L-2m}
	Let $\left( H_{n} \right)$ be a PLRS with $L$ coefficients defined by  $[1,0,\dots, 0, \underbrace{1,\dots, 1}_{m}, N]$, where $L \geq 2m + 2$. Then $\left( H_{n} \right)$ is complete if and only if \begin{equation}
N\leq \left\lfloor \frac{\left( L-m \right) \left( L+m+1 \right) }{4}+\frac{1}{48}m(m+1)(m+2)(m+3)+\frac{1-2m}{2} \right\rfloor.\end{equation}
\end{theorem}
\begin{proof}
First, note for all $1\leq n\leq L-m$, that $H_{n}=n$.

Now, we claim that for all $1\leq k\leq m$, \begin{equation}
	H_{L-m+k}=L-m+\frac{1}{6}k(k+1)(k+2)+k
.\end{equation} 
We use induction, appealing to the identity $\sum_{a=1}^{n}a(a+1)/{2}=n(n+1)(n+2)/6$. We first see that 
\begin{equation}
	H_{L-m+1}=H_{L-m}+H_1+1=L-m+2 = L-m + \sum_{a=1}^{1}\frac{a(a+1)}{2}+1.
\end{equation}
Additionally, \begin{equation}
	H_{L-m+2}=H_{L-m+1}+H_2+H_1+1 = (L-m+2)+2+1+1 = L-m + \sum_{a=1}^{2}\frac{a(a+1)}{2}+2
.\end{equation} 
Now, suppose $H_{L-m+k}=L-m+\sum_{a=1}^{k}a(a+1)/{2}+k$ for some $k<m$. Note that \begin{equation}
H_{L-m+k+1}=H_{L-m+k}+H_{k+1}+\dots+H_{1}+1.\end{equation}
Since we supposed $L \geq 2m+2$, we see $k+1 \leq m+1 \leq L-m$, and thus for all $ 1\leq i\leq k,\; H_{i}=i$. Thus,
\begin{align}
 H_{L-m+k+1} &= \left( L-m+\sum_{a=1}^{k}\frac{a(a+1)}{2}+k \right) +\frac{(k+1)(k+2)}{2}+1 \nnend
 &= L-m +\sum_{a=1}^{k+1}\frac{a(a+1)}{2}+k+1.
\end{align}

Thus, we have an explicit formula for $H_{i}$, for $1\leq i\leq L$. 

Note that $\left( H_{n} \right)$ is complete if and only if it fulfills Brown's criterion for the $(L+1)$st and $(L+2)$nd term. We show that $\left( H_{n} \right)$ fulfills the criterion for $L+2$ if and only if the bound above holds; it is not difficult to show that the bound for $L+1$ is less strict.

Indeed, we wish to reduce the inequality
\begin{align}
    H_{L+2}=H_{L+1}+H_{m+2}+\dots+H_{3}+2N &\leq 1+\sum_{i=1}^{L+1}H_{i}\\
	\iff H_{m+2}+\dots+H_{3}+2N &\leq 1+\sum_{i=1}^{L}H_{i}\label{1}.
\end{align}

Simplifying the left hand side of inequality \eqref{1}, 
\begin{align}
		H_{m+2}+\dots+H_{3}+2N &=H_{m+2}+\dots+H_3+(H_2+H_1 - H_2 - H_1) +2N\nnend&= \frac{(m+2)(m+3)}{2}-3+2N.
\end{align}
Additionally, 
\begin{align}
	1+\sum_{n=1}^{L}H_{n}&=1+\sum_{n=1}^{L-m}H_{n}+\sum_{n =L-m+1}^{L}H_{n}\nnend
	&=1 + \frac{(L-m)\left( L-m+1 \right) }{2}+\sum_{n=1}^{m}\left( \frac{1}{6}n(n+1)(n+2) +n+L-m\right).\label{eq:351}
\end{align}
We use the fact that $\sum_{n=1}^{m}n(n+1)(n+2)=m(m+1)(m+2)(m+3)/4$ to simplify \eqref{eq:351} as follows:
\begin{multline}
1+\frac{(L-m)\left( L-m+1 \right) }{2}+\frac{m(m+1)}{2}+mL-m ^2+\frac{1}{6}\sum_{n=1}^{m}n(n+1)(n+2)\\
=1+ \frac{(L-m)\left( L-m+1 \right) }{2}+\frac{m(m+1)}{2}+mL-m ^2+\frac{1}{24}m(m+1)(m+2)(m+3)
.\end{multline}
Hence \eqref{1} is equivalent to
\begin{multline}
	\frac{\left( m+2 \right) \left( m+3 \right) }{2}-3+2N \leq 1 +  \frac{(L-m)\left( L-m+1 \right) }{2}+\frac{m(m+1)}{2}\\
	+mL-m ^2+\frac{1}{24}m(m+1)(m+2)(m+3).
\end{multline}

Simplifying, this gives us \begin{equation}
N\leq \left\lfloor \frac{\left( L-m \right) \left( L+m+1 \right) }{4}+\frac{1}{48}m(m+1)(m+2)(m+3)+\frac{1-2m}{2} \right\rfloor.\end{equation} 
\end{proof}

\begin{theorem}\label{sum2L-2m 2nd}
Let $\left(G_n\right)$ and $\left(H_n\right)$ be PLRS's, both with $L$ coefficients, which are defined by $[1, 0, \dots, 0, \underbrace{1, \dots, 1}_{m}, N]$ and $[1, 0, \dots, 0, \underbrace{1, \dots, 1}_{m+1}, N+1]$ respectively. Suppose $L - m \geq 4$ (so that at least one zero is present in $\left(H_n\right)$), $m \geq (L-1) / 2$, and $\left(G_n\right)$ is complete. Then $\left(H_n\right)$ is also complete.
\end{theorem}
\begin{proof}
As $\left(G_n\right)$ is complete, from Brown's criterion, we obtain
\begin{equation}
G_{L+2} = G_{L+1} + \sum_{i=3}^{m+2}G_i + NG_2 \leq 1 + \sum_{i=1}^{L+1}G_i,
\end{equation}
which is equivalent to 
\begin{equation}\label{baseG}
2N \leq \sum_{i = m+3}^L G_i + 4.
\end{equation}
From Lemma~\ref{firstL+2}, it suffices to show that
\begin{equation}
H_{L+1} \leq 1 + \sum_{i=1}^L H_i \qquad\text{and}\qquad H_{L+2} \leq 1 + \sum_{i=1}^{L+1} H_i,
\end{equation}
or equivalently,
\begin{equation}
    N \leq \sum_{i=m+3}^{L-1}H_i\label{easybound}
\end{equation}
and
\begin{equation}
    2N \leq \sum_{i=m+4}^{L}H_i + 2\label{hardbound}.
\end{equation}
We first show \eqref{easybound}. Combining with \eqref{baseG}, it suffices to show that 
\begin{equation}
\sum_{i=m+3}^L G_i + 4 \leq 2\sum_{i=m+4}^L H_i.
\end{equation}
From Lemma~\ref{diffGH}, 
\begin{equation}
\begin{cases}
    G_i \leq H_i, & \text{if $m + 3 \leq i \leq L$;}\\
    G_i \leq H_{i-1} - 1, & \text{if $2(L-m) < i \leq L$}.
\end{cases}
\end{equation}
Thus, 
\begin{align}
    \sum_{i=m+3}^L G_i + 4 &= \sum_{i=m+3}^{2(L-m)}G_i + \sum_{i=2(L-m)+1}^{L}G_i + 4 \nnend
    & \leq \sum_{i=m+3}^{2(L-m)}H_i + \sum_{i=2(L-m)}^{L-1}H_i + (2m - L + 4) \nnend
    & \leq 2\sum_{i=m+4}^L H_i,
\end{align}
where the last inequality can be taken crudely.
We then show \eqref{hardbound}. Similarly, combining with \eqref{baseG}, it suffices to show that
\begin{equation}
\sum_{i=m+3}^L G_i + 2 \leq \sum_{i=m+4}^L H_i.
\end{equation}
If $m + 3 \geq 2(L-m)$, then 
\begin{align}
    \sum_{i=m+4}^LH_i &= \sum_{i=m+4}^L\left(H_{i-1} + \sum_{j=1}^{i-L+m+1}H_j + 1\right) \nnend
    &\geq \sum_{i=m+4}^L(H_{i-1} + H_{i-L + m + 2}) \text{ (Brown's criterion for the first terms)} \nnend
    &= \sum_{i=m+3}^{L-1} H_i + \sum_{i=2m+6-L}^{m+2}H_i \geq \sum_{i=m+3}^{L-1}(G_{i+1} + 1) + H_{m+2} \nnend
    &\geq \sum_{i=m+3}^L G_i + 2. 
\end{align}
If $m + 3 < 2(L-m)$, then 
\begin{align}
    \sum_{i=m+3}^L G_i &= \sum_{i=m+3}^{2(L-m)-1}G_i + G_{2(L-m)} + \sum_{i=2(L-m) +1}^L G_i \nnend
    &= \sum_{i=m+3}^{2(L-m)-1}(H_{i-1} + 1) + H_{2(L-m)-1} + \sum_{i=2(L-m) +1}^L G_i \nnend
    &= \sum_{i=m+2}^{2(L-m)-1}H_i + (2L - 3(m+1)) + \sum_{i=2(L-m) +1}^L G_i.
\end{align}
Thus, our original inequality \eqref{hardbound} holds if we can show that
\begin{equation}
H_{m+2} + H_{m+3} + (2L - 3(m+1)) + \sum_{i=2(L-m) +1}^L G_i \leq \sum_{i=2(L-m)}^L H_i.
\end{equation}
Similarly to the previous case,
\begin{align}
    \sum_{i=2(L-m)}^L H_i &\geq \sum_{i=2(L-m)}^L(H_{i-1} + H_{i-L+m+2}) \nnend
    &= \sum_{i=2(L-m)-1}^{L-1}H_i + \sum_{i=L-m+2}^{m+2}H_i \nnend 
    &= \sum_{i=2(L-m)}^{L-1}H_i + H_{2(L-m)-1} + H_{m+2} + \sum_{i=L-m+2}^{m+1}H_i.
\end{align}
As $2(L-m) - 1 \geq m+3$ and $H_i \geq i$,
\begin{align}
    \sum_{i=2(L-m)}^L H_i &\geq \sum_{i=2(L-m)}^{L-1}(G_{i+1} + 1) + H_{m+3} + H_{m+2} + \sum_{L-m+2}^{m+1}i \nnend
    &= \sum_{i=2(L-m)+1}^L G_i + H_{m+3} + H_{m+2} + \left( 2m - L + \sum_{i=L-m+2}^{m+1}i \right).
\end{align}
From Lemma~\ref{lem:trivialineq},
\begin{equation}
    \sum_{i=2(L-m)}^L H_i \geq \sum_{i=2(L-m)+1}^L G_i + H_{m+3} + H_{m+2} + (2L - 3(m+1)).
\end{equation}
\end{proof}

\section{An analytical approach}\label{roots}

\subsection{An introduction to principal roots}
We begin by restating some results from Martinez, Miller, Mizgerd, Murphy, and Sun \cite{MMMMS}.

\begin{lemma}\label{lma:principalRoot}
Let $P(x)$ be the characteristic polynomial of a recurrence relation with nonnegative coefficients and at least one positive coefficient. Let $S = \{m \ |\ c_m \neq 0\}$. Then
\begin{enumerate}
    \item there exists exactly one positive root $r$, and this root has multiplicity $1$,
    \item every root $z \in \C$ satisfies $|z| \leq r$, and
    \item if $\gcd(S) = 1$, then $r$ is the unique root of greatest magnitude.
\end{enumerate}
\end{lemma}

\begin{proof}
This is Lemma 2.1 from Martinez, Miller, Mizgerd, Murphy, and Sun \cite{MMMMS}.
\end{proof}

\begin{remark}
We refer to the unique positive root from Lemma~\ref{lma:principalRoot} as the \emph{principal root} of the recurrence sequence and corresponding characteristic polynomial.
\end{remark}

\begin{lemma}
Let $P(x)$ be the characteristic PLRS $\left(H_n\right)$ and let $r_1$ be its principal root. Then
\begin{equation}
\lim_{n \to \infty} \frac{H_n}{r_1^n} = C
\end{equation}
for some constant $C > 0$.
\end{lemma}

\begin{proof}
Corollary 2.3 from \cite{MMMMS} proves a stronger result than this, which immediately implies this lemma.
\end{proof}

\begin{lemma}\label{lma:principalRootContributes}
Let $P(x)$ be the characteristic polynomial of a PLRS $\left(H_n\right)$ with roots $r_i$, each of multiplicity $m_i$, where $r_1$ is the principal root. If
\begin{equation}\label{eqn:explicitHn}
H_n = a_1 r_1^n + \sum_{i=2}^k q_i(n)r_i^n, 
\end{equation}
where $q_i(x)$ is a polynomial of degree at most $m_i - 1$, then $a_1 > 0$.
\end{lemma}

\begin{proof}
First, note that the set $S$ of Lemma~\ref{lma:principalRoot} contains $1$ because $c_1 > 0$ in a PLRS. Therefore $\gcd(S) = 1$, and $r_1$ is the unique root of greatest magnitude. If $a_1 < 0$, then this implies that $H_n < 0$ for some $n$ because the behavior of $a_1 r_1^n$ eventually dominates the expression for $H_n$ in (\ref{eqn:explicitHn}). If $a_1 = 0$, then
\begin{equation}
\lim_{n \to \infty} \frac{H_n}{r_1^n} = 0
\end{equation}
because $r_1$ is the unique root of greatest magnitude, so if $a_1 = 0$ then the behavior of $H_n$ is bounded by geometric growth of the root of next greatest magnitude, which is necessarily smaller than $r_1^n$. Thus, $a_1 > 0$.
\end{proof}

\subsection{Applications to completeness}
Given these results, we see that the principal root of a PLRS serves as a measure for the rate of that sequence's growth. Guided by the simple heuristic that, generally, a sequence which grows slowly is more likely to be complete than a sequence which grows rapidly, we find bounds for the potential roots of a complete or incomplete PLRS. We aim to answer these questions: for any given $L$, what is the fastest-growing complete PLRS with $L$ coefficients? What is the slowest-growing incomplete PLRS with $L$ coefficients? While the principal root of a PLRS has not been related to completeness before, there is previous work on bounding the principal root of other linear recurrence sequences by Gewurz and Merola \cite{GM}.

\begin{lemma}\label{thm:CompleteCriterionRoots}
If $\left(H_n\right)$ is a complete PLRS and $r_1$ is its principal root, then $|r_1| \leq 2$.
\end{lemma}

\begin{proof}
Suppose that $|r_1| > 2$. Set
\begin{equation}
	H_n = a_1r_1^{n}+q_2 (n)r_2^n + \cdots + q_r (n)r_k^n.
\end{equation}
Since $r_1$ is the unique root of largest magnitude by Lemma~\ref{lma:principalRoot}, the behavior of $a_1 r_1^n$ dominates in the limit. By Lemma~\ref{lma:principalRootContributes}, $a_1 > 0$, so if $|r_1| > 2$, then eventually $|a_1 r_1^n| > 2^{n - 1}$, and so there exists a large $n$ for which $H_{n}>2^{n-1}$. As the sequence $\left( 2^{n-1} \right)$ is the complete PLRS with maximal terms by Theorem~\ref{clm:largestCompleteGaps}, we see $\left( H_{n} \right)$ must be incomplete.
\end{proof}

\begin{remark}
The converse to this lemma does not hold. A counterexample is $[1, 1, 1, 0, 4]$, which has principal root 2 but is not complete.
\end{remark}

While the proof is simple, this lemma gives us an effective upper bound for the roots of a complete PLRS, regardless of length. Recall from Theorem~\ref{basic} that for any $L$, the PLRS $\left( H_{n} \right)$ generated by the coefficients $[\underbrace{1,\ldots , 1}_{L-1},2]$ satisfies $H_{n}=2^{n-1}$. This sequence naturally has a principal root of 2, and is complete. Similarly, for any $L \geq 1$, the sequence $[\underbrace{1,\ldots , 1}_{L}]$ is complete, and its principal root asymptotically approaches 2 as $L$ grows.

We now focus on finding a lower bound for the roots of an incomplete sequence, which proves to be a more difficult problem.

\begin{lemma}\label{boundexists}
For any $ L\in \Z_{>0}$, there exists a constant $ B_{L}$, with $  1< B_{L}<2$ such that if $\left( H_{n} \right)$ is a PLRS with principal root $r_1$ and $r_1<B_{L}$, then $\left( H_{n} \right)$ is complete.
\end{lemma}

\begin{remark}
    This means that for any $L$, there exists a lower bound $B_{L}$ on possible values of the principal root of an incomplete PLRS generated by $[c_1,\ldots, c_{L}]$.
\end{remark}

\begin{proof}
	In order to show that such a $B_{L}$ exists, it suffices to show that for any given $L$, there exists only finitely many incomplete positive linear recurrence sequences generated by $[c_1,\dots, c_{L}]$ with principal root $r_1 <2$.

	Recall that the principal root $r_1$ of a PLRS is the single positive root of the characteristic polynomial $p(x)=x^{L}-\sum_{i=1}^{L}c_{i}x^{L-i}$. As $\lim_{x \rightarrow \infty }p(x)=+\infty $, the fact that $r_1$ is the unique positive root of $p(x)$ implies that $r_1<2 \iff p(2)>0$, by IVT. Note that
	\begin{equation}
		p(2)=2^{L}-\sum_{i=1}^{L}c_{i}2^{L-i}>0 \iff \sum_{i=1}^{L}c_{i}2^{L-i}<2^{L}.
	\end{equation} 
	As for all $i$, $c_{i}\geq 0$, so the inequality above cannot hold if there exists $i$ such that $c_{i}\geq 2^{i}$. As the set $\{ [c_1,\dots, c_{L} ]: 0\leq c_{i}\leq 2^{i} \text{ for all } i\}$ of such sequences is finite, we are done.
\end{proof}

The remainder of this section is a series of lemmas which build towards the following conjecture.
\begin{conjecture}\label{final}
	Let $N_{L}=\left\lceil L(L+1)/4 \right\rceil$, and let $\lambda_{L} $ be the principal root of the sequence generated by $[1,\underbrace{0,\ldots , 0}_{L-2}, N_{L} +1]$, i.e., the sole principal root of
    \begin{equation}
	p_{L}(x)=x^{L}-x^{L-1}-\left\lceil \frac{L(L+1)}{4} \right\rceil -1.
    \end{equation}
	Then, if $[c_1,\ldots , c_{L}]$ generates an incomplete sequence, then its principal root is at least $\lambda_{L}$; in other words, the incomplete sequence of length $L$ with the smallest possible principal root is precisely $[1,\underbrace{0,\ldots , 0}_{L-2}, N_{L} +1]$. 
\end{conjecture}

\begin{remark}
This conjecture is equivalent to stating $B_{L}=\lambda _{L}$ for all $L\geq 2$, where $B_{L}$ is the bound proposed in Lemma~\ref{boundexists}.
\end{remark}

\begin{remark}
By using Theorem~\ref{thm:1onekzero}, it is easy to see that the sequence generated by $[1,0,\ldots , 0, N_{L} +1]$ is incomplete; in fact, the value $N_{L}+1$ is the minimal positive integer such that a sequence of this form is incomplete.
\end{remark}

As a first step towards a proof of Conjecture~\ref{final}, we prove Lemma~\ref{minprincipalroot}, which addresses the case of sequences with a large sum in coefficients.

\begin{definition}
For positive integers $S,L$, we define the set of positive linear recurrence sequences 
\begin{equation}
P_{L,S} \coloneqq \biggl\{ \left(H_n\right) \text{ generated by } [c_1,\ldots, c_{L}]\ \bigg| \  \sum_{i=1}^{L}c_{i}=S+1  \biggr\}.
\end{equation} 
\end{definition}

\begin{lemma}\label{PLSRoots}
The sequence in $P_{L,S}$ with the minimal principal root is $[1,0,\ldots, 0,S]$.
\end{lemma}

\begin{proof}
Consider a sequence generated by $s=[c_1,\ldots, c_{L}]\in P_{L,S}$, and let $r_1,\ldots, r_{L}$ be its roots, with $r_1>0$ the principal root. Since $|c_{L}|=\bigl|  \prod_{i=1}^{L}r_{L}\bigr|$ is a positive integer, we know $r_1>1$. 
Now, for any $1\leq m\leq L$ consider a sequence generated by $s_{m} \in P_{L,S}$ of the form
\begin{equation}
	[c_1,\ldots,c_{m-1}, c_{m}-1,c_{m+1},\ldots, c_{L}+1].
\end{equation} 
We claim that the principal root $q_1$ of $s_{m}$ fulfills $q_1<r_1$.

Define the characteristic polynomials $f(x)$ and $g(x)$ for $s$ and $s_{m}$, respectively, so that
\begin{equation}
f(x)=x^{L}-\sum_{i=1}^{L}c_{i}x^{L-i},
\end{equation}
and
\begin{align}
    g(x) &= x^{L}-\sum_{i=1}^{m-1}c_{i}x^{L-i}-\left( c_{m}-1 \right)x^{m} - \sum_{i=m+1}^{L-1}c_{i}c_{i}x^{L-i}-\left( c_{L}+1 \right)\nnend
    &=x^{L}-\sum_{i=1}^{L}c_{i}x^{L-i}+x^{m}-1.
\end{align}

As $q_1$ is the sole positive root of $g(x)$, and $g(x)$ is eventually positive, we notice that $q_1<r_1$ if and only if $g(r_1)>0$, which is equivalent to $g\left( r_1 \right) >f(r_1)$.

Now, 
\begin{equation}
\begin{array}{l@{{}\iff {}}l}
	g\left( r_1 \right) >f\left( r_1 \right) 
& r_1^{L}-\sum_{i=1}^{L}c_{i}r_1^{L-i}+r_1^{m}-1>r_1^{L}-\sum_{i=1}^{L}c_{i}r_1^{L-1} \\ 
&  r_1^{m}-1 >0 \\
 & r_1>1. \\
\end{array}
\end{equation}

As $r_1 > 1$, the principal root $q_1$ of $g(x)$ is strictly less than that of $f(x)$.

As $s$ was chosen arbitrarily, we see that the principal root of any sequence $s\in P_{L,S}$ can be strictly decreased by using the transformation $s\rightarrow s_{m}$ for any $1\leq m\leq L$. Applying this transformation iteratively for all values of $m$, we inevitably end up with the minimal possible values of $c_1,\ldots, c_{L-1}$, namely $c_1=1,\; c_2=c_3=\cdots =c_{L-1}=0$, and the maximal possible value of $c_{L}$, namely $c_{L}=S$. 

Thus, as the principal root under these iterated transformations is strictly decreasing, we conclude that $[1,0,\ldots, 0,S]$ has the smallest principal root of any element of $P_{L,S}$. 
\end{proof}

\begin{lemma}\label{DecreaseLastDecreasesRoot}
For any $S > 0$, the principal root of $[1,0,\ldots,0,S]$ is strictly less than that of $[1,0,\dots,0,S+1]$. 
\end{lemma}

\begin{proof}
Let $S$ be an arbitrary positive integer, and let $f(x), g(x)$ and $r_1,q_1$ denote the characteristic polynomials and principal roots of $[1,0,\ldots, 0,S+1]$ and $[1,0,\ldots, 0,S]$, respectively. 

As before,  $q_1<r_1$ if and only if $g(r_1)>0=f(r_1)$. Note that 
\begin{equation}
\begin{array}{l@{{}\iff {}}l}
	g(r_1)>f(r_1)
& r_1^{L}-r_1^{L-1}-S>r_1^{L}-r_1^{L-1}-\left( S+1 \right)  \\
& S+1>S.
\end{array}
\end{equation}
Thus, $q_1<r_1$, for any value of $S$.
\end{proof}

\begin{lemma}\label{minprincipalroot}
	Any sequence fulfilling $\sum_{i=1}^{L}c_{i}\geq N_{L}+2$ has a principal root greater than or equal to that of \begin{equation} [1,0,\ldots,0,N_{L}+1]. \end{equation}
\end{lemma}

\begin{proof}
Recall from Theorem~\ref{thm:1onekzero} that the sequence $[1,0,\dots,0,N] $ is complete if and only if $N\leq N_{L}$, for $N_L= \left\lceil L(L+1)/{4}\right\rceil$. Thus, an immediate corollary to Theorem~\ref{thm:1onekzero} is that the incomplete sequence of the form $[1,0,\dots,0,N]$ with the minimal possible principal root is $[1,0,\ldots,0,N_{L}+1]$.

Furthermore, if we have a sequence generated by $[c_1,\ldots, c_{L}]$ which fulfills $\sum_{i=1}^{L}c_{i}\geq N_{L}+2$, Lemmas~\ref{PLSRoots} and \ref{DecreaseLastDecreasesRoot} present a sequence of algorithms which allow us to transform this sequence into the sequence generated by $[1,0,\ldots,0,N_{L}+1]$, in such a way that each transformation strictly lowers the magnitude of the principal root. 

Thus, any sequence satisfying $\sum_{i=1}^{L}c_{i}\geq N_{L}+2$ has a principal root strictly greater than the principal root of $[1,0,\ldots,0,N_{L}+1]$.
\end{proof}

The following lemmas are working towards proving Conjecture~\ref{finalsegundamitad}, which addresses the second case of Conjecture~\ref{final}, which addresses the roots of sequences $[c_1, \ldots, c_L]$ which fulfill $\sum_{i=1}^L c_i \leq N_L +2$. 

\begin{lemma}\label{addingcoeffdecreasesroot}
	Suppose the sequence generated by $[c_1,\ldots, c_{L}]$ has principal root $r$, then for any $c_{L+1}\in \Z_{>0}$, the sequence generated by $[c_1,\ldots, c_{L},c_{L+1}]$ (in which we add an additional positive coefficient) with principal root $q$ fulfills $r<q$. 
\end{lemma}

\begin{proof}
Let $f(x),g(x)$ be the characteristic polynomials of the two sequences, so that
\begin{equation}
f(x)=x^{L}-\sum_{i=1}^{L}c_{i}x^{L-i} \quad \text{and} \quad g(x)=x^{L+1}-\sum_{i=1}^{L}c_{i}x^{L+1-i}-c_{L+1}.
\end{equation} 
Similar to previous arguments, by the IVT, $r<q$ if and only if $g(r)<f(r)=0$. Note that
\begin{equation}
\begin{array}{l@{{}\iff {}}l}
	g(r)<f(r)
& r^{L+1}-\sum_{i=1}^{L}c_{i}x^{L+1-i}-c_{L+1}<r^{L}-\sum_{i=1}^{L}c_{i}r^{L-i} \\
& c_{L+1} > r^{L+1}-r^{L}+\sum_{i=1}^{L}c_{i}r^{L-i}-\sum_{i=1}^{L}c_{i}r^{L+1-i} \\
& c_{L+1}>r^{L}\left( r-1 \right) +\sum_{i=1}^{L}c_{i}r^{L-i}\left( 1-r \right)  \\
& c_{L+1}>\left( 1-r \right) \left( r^{L}-\sum_{i=1}^{L}c_{i}r^{L-i} \right) =\left( 1-r \right) \cdot f(r)   \\
& c_{L+1}> \left( 1-r \right) f(r)=\left( 1-r \right) \cdot 0	=0. \\
\end{array}
\end{equation}
Since $c_{L+1}\in \Z_{>0}$, the last line holds. It follows immediately that $r<q$.
\end{proof}

\begin{lemma}\label{rootbounddecreases}
Let $\lambda _{L}$ be the principal root of \begin{equation}
x^{L} - x^{L - 1} - N_L - 1.
\end{equation} 
Then, for any $L\geq 2$, $\lambda _{L}>\lambda _{L+1}$.
\end{lemma}

\begin{proof}
	We let $f(x)$ and $g(x)$ denote the characteristic polynomials of $[1,0,\ldots, 0,N_{L}+1]$ and $[1,0,\ldots, 0,N_{L+1}+1]$, of length $L$ and $L+1$, respectively. This way we obtain
	\begin{equation}
	f(x)=x^{L}-x^{L-1}-N_{L}-1,\; \; \; g(x)=x^{L+1}-x^{L}-N_{L+1}-1.
	\end{equation} 
	As in previous proofs, we see that $\lambda _{L}>\lambda _{L+1}\iff g\left( \lambda _{L} \right) >f\left( \lambda _{L} \right) =0$.
\begin{equation}
\begin{array}{l@{{}\iff {}}l}
	g\left( \lambda  \right) >f\left( \lambda  \right) 
& \lambda ^{L+1}-\lambda ^{L}-N_{L+1}-1> \lambda ^{L}-\lambda ^{L-1}-N_{L}-1 \\
& \lambda ^{L+1}-2\lambda ^{L}+\lambda ^{L-1}>N_{L+1}-N_{L} \\
&\lambda ^{L-1}\left( \lambda -1 \right) ^2 > N_{L+1}-N_{L}. \\
\end{array}
\end{equation}
Note that when $f\left( \lambda  \right) =0$, we have $\; \lambda ^{L-1}\left( \lambda -1 \right) =N_{L}+1$. Moreover, $N_{L+1}-N_{L}\leq (L+2)/{2}$, which can be shown by using the definition of $N_L$ and checking all cases modulo 4. Thus, it suffices to show that \begin{equation}
	\left( N_{L}+1 \right) \left( \lambda _{L}-1 \right) \geq \frac{L+2}{2}.
\end{equation} 
Now, using the value of $N_{L}$, all we need to show is
\begin{equation}\label{appendixit}
	\left( \lambda _{L}-1 \right) \geq \frac{L+2}{L^2+L+4}. 
\end{equation}
The proof of \eqref{appendixit} is just algebra, and is left to Appendix~\ref{apx:sec4lemmas}, as Lemma~\ref{appendixitAppendix}.
\end{proof}

\begin{lemma}\label{rootsgotozero}
For any $L \in \N$, let $\lambda _{L}$ be the sole positive root of the polynomial
\begin{equation}
	p_{L}(x)=x^{L}-x^{L-1}-\left\lceil \frac{L(L+1)}{4} \right\rceil -1.
\end{equation} 
Then, $\lim_{L \rightarrow \infty }\lambda _{L}=1$.
\end{lemma}

\begin{proof}
	We show that for any $\varepsilon >0$, there exists an $M$ large enough so that for all $L > M$, $p_{L}(1+\varepsilon )>0$. As $p_{L}(x)$ has only one positive root $\lambda _{L}$ and $p(x)$ is positive as $x\rightarrow \infty $, we see $p_{L}(1+\varepsilon )>0$ implies $\lambda _{L}<1+\varepsilon $. If this is possible for arbitrary $\varepsilon$, then $\lambda _{L}\rightarrow 1$ as desired.

Let us fix an $\varepsilon >0$. For any $L$, we may write 
\begin{align}
	p_{L}(1+\varepsilon )&=\left( 1+\varepsilon  \right) ^{L}-(1+\varepsilon )^{L-1}-\left\lceil \frac{L(L+1)}{4} \right\rceil -1 \nnend
	&= \sum_{n=0}^{L}\varepsilon ^{n}\left( \binom{L}{n}-\binom{L-1}{n} \right) - \left\lceil \frac{ L(L+1) }{4} \right\rceil -1,\label{uglyformpl}
    \end{align}
where $\binom{L-1}{L}$ is $0$. Using Pascal's rule ($ \binom{n-1}{k} + \binom{n-1}{k-1} = \binom{n}{k}$), we can reduce \eqref{uglyformpl} to 
\begin{equation}\label{eq:Pascalrule}
	p_{L}\left( 1+\varepsilon  \right) = \sum_{n=1}^{L}\varepsilon ^{n}\binom{L-1}{n-1} - \left\lceil L(L+1)/4 \right\rceil -1.
\end{equation}
The quantity from (\ref{eq:Pascalrule}) can easily be shown to be positive (and in fact tends towards infinity) for large enough $L$. For example, we can take the trivial bound 
\begin{equation}
\sum_{n=1}^{L}\varepsilon ^{n}\binom{L-1}{n-1} > \varepsilon ^4\binom{L-1}{3},
\end{equation} 
as the full sum must be larger than only its fourth summand.

Since $\varepsilon ^{4}$ is simply a positive constant and $L(L+1) \ll \binom{L-1}{3}$, then for large enough $L$,
\begin{equation}
	p_{L}(1+\varepsilon )>\varepsilon ^4\binom{L-1}{3}-\left\lceil L(L+1)/4 \right\rceil -1 >0.
\end{equation} 
\end{proof}

\begin{remark}
	Even in the event that Conjecture~\ref{final} is false, this gives us conclusive proof that we may find incomplete sequences whose roots are arbitrarily close to 1. Since 1 is the minimum possible size for the root of a PLRS, this may be interpreted as proof that we may find arbitrarily slow-growing incomplete sequences, with coefficients of any length $L$.
\end{remark}

\begin{lemma}\label{addingm}
	Consider the sequence generated by $[c_1,\ldots,c_{L}]$. For any value $m \in \Z_{>0}$, the principal root of $[c_1,\ldots, c_{L}+m]$ is greater than that of $[c_1,\ldots, c_{L},m]$. 
\end{lemma}
	
\begin{proof}
	Let $f(x),g(x)$ be the characteristic polynomials and $r,q$ be the principal roots of $[c_1,\ldots, c_{L}+m]$ and $[c_1,\ldots, c_{L},m]$, respectively. Since each of $f$ and $g$ has a unique positive root, we see that $r > q\iff g(r)>f(r)=0$. Note that 
\begin{equation}
\begin{array}{l@{{}\iff {}}l}
	g(r)>0
& 0=rf(r)<g(r) \\
& r\left(  r^{L}-\sum_{i=1}^{L}c_{i}r^{L-i}-m\right) < r^{L+1}-\sum_{i=1}^{L}c_{i}r^{L+1-i}-m \\
 & m<rm \\
 & r>1.
\end{array}
\end{equation}
Thus, the inequality always holds, and so $r > q$, as desired.
\end{proof}

\begin{conjecture}\label{finalsegundamitad}
Let $\lambda_{L} $ be the principal root of $x^{L}-x^{L-1}-N_{L}-1$. If the sequence generated by $[c_1,\ldots, c_{L}]$ is incomplete with $\sum_{i=1}^{L}c_{i}\leq \left\lceil L( L+1)/ {4}\right\rceil+2$, then its principal root is at least $\lambda_{L} $. 
\end{conjecture}

We present a partial proof, which addresses all cases except what is denoted as Subcase 2.

\begin{proof}[Partial proof.]
We use induction.

For $L=2$, $N_{L}=\left\lceil 2\cdot 3/{4}\right\rceil=2$, and so the coefficients $[c_1,c_2]$ fulfilling the requirement are of the form $c_1+c_2\leq  4$. The incomplete sequences of this form have coefficients $[2,1]$, $[2,2]$, $[1,3]$, and $[3,1]$. Checking each case directly, we see that their principal roots are approximately $2.414$, $2.731$, $2.303$, and $3.303$, respectively. Among these roots, the root of $[1,3]=[1,N_2+1]$ is the minimum; thus, the lemma holds for the base case.

Now, suppose the lemma holds for some value of $L\geq 2$. We show that the Lemma holds for $L+1$ as well.

Let $[c_1,\ldots, c_{L},c_{L+1}]$ be an incomplete sequence with $\sum_{i=1}^{L+1}c_{i}\leq \left\lceil (L+1)(L+2) /{4}\right\rceil+2$.

\begin{enumerate}[label={\textbf{Case \arabic*:}}, leftmargin=*]
\item  $\sum_{i=1}^{L}c_{i}< N_{L}+2$.

Under the condition above, the following two sub-cases arise.

\begin{enumerate}[label={\textbf{Sub-Case \arabic*:}}, leftmargin=*]

\item $[c_1,\ldots, c_{L}]$ is incomplete.

If the sequence is incomplete, then, by our inductive hypothesis, since $\sum_{i=1}^{L}c_{i}\leq N_{L}+2$, we must have that the principal root $r$ of $[c_1,\ldots, c_{L}]$ is greater than or equal to $\lambda _{L}$. Hence, by Lemma~\ref{addingcoeffdecreasesroot}, since the principal root $q$ of $[c_1,\ldots, c_{L+1}]$ satisfies $q>r$, we have that $[c_1,\ldots, c_{L+1}]$ has principal root $q>\lambda_{L} $. Finally, by Lemma~\ref{rootbounddecreases}, we know that $\lambda _{L}>\lambda _{L+1}$. Therefore, we have $q>r \geq \lambda _{L}>\lambda _{L+1}$, and the statement holds in this case.

\item $[c_1,\ldots, c_{L}]$ is complete:

The proof of this sub-case has not been found yet, hence why the statement remains a conjecture.
\end{enumerate}

\item $\sum_{i=1}^{L}c_{i}\geq  N_{L}+2$.

If this inequality holds, the transformations developed in Lemmas~\ref{PLSRoots} and \ref{DecreaseLastDecreasesRoot}, imply that $[c_1,\ldots, c_{L}]$ has principal root of at least $\lambda$. Applying Lemma~\ref{addingcoeffdecreasesroot}, we see that the principal root of $[c_1,\dots, c_{L+1}]$ is strictly greater, and thus the statement holds in this case.\qedhere
\end{enumerate}
\end{proof}

The results in this section provide us with an efficient way to verify completeness for PLRS's. Namely, for a sequence $[c_1, \ldots, c_L]$, we may evaluate its characteristic polynomials at the points $B_L$ and $2$, which provides the following information.
\begin{itemize}
\item If $p(2)<0$, the sequence is incomplete.
\item If $p(B_L)>0$, the sequence is complete.
\item If $p(2)\geq 0$ and $p(B_L)\leq 0$, then the principal root of the sequence lies in the interval $[B_L, 2]$, and so further inquiry is necessary to determine whether the sequence is complete.
\end{itemize}

Computationally, evaluating a polynomial of degree $L$ is an $\mathcal{O}(L^2)$ problem. Generating a minimum of $2L$ terms of the sequences and checking Brown's criterion for each, on the other hand, is a $\mathcal{O}(2^L)$ problem. Thus, this method---even if inconclusive---provides fast and efficient method to categorize sequences, and narrows our search to the interesting interval $[B_L, 2]$, in which both complete and incomplete sequences arise.

\subsection{Denseness of incomplete roots}

Having narrowed our search for principal roots of complete and incomplete sequences to the interval $[B_L, 2]$, it is only natural to ask how the roots of these sequences are distributed throughout the interval. 

\begin{lemma}\label{rootsorder}
	For fixed $L>2$ and $k>0$, define the three polynomials $f(x)=x^{L}-x^{L-1}-k$, $g(x)=x^{L}-x^{L-1}-\left( k+1 \right) $, and $h(x)=x^{L}-x^{L-1}-\left( k+2 \right)$. Let $q,r$, and $s$ be the sole positive roots of $f,g$, and $h$ respectively, so that $1<q<r<s$. Then, 
\begin{equation}
r-q>s-r.
\end{equation}
\end{lemma}

\begin{proof}
From the definition of $f,g$, and $h$, we see that 
\begin{align}
q^{L}-q^{L-1}&= k,\nnend
r^{L}-r^{L-1}&= k+1,\nnend
s^{L}-s^{L-1} &=k+2.\label{diferenciasderoots}
\end{align}

Now, define the polynomial $p(x)=x^{L}-x ^{L-1} $. Taking the first and the second derivative of $p$, we see $p'(x)=Lx^{L-1}-\left( L-1 \right) x^{L-2}$, and $p''(x)=L\left( L-1 \right) x^{L-2}-\left( L-1 \right) \left( L-2 \right) x^{L-3}$. In particular, for all $x\geq 1$, $p(x)\geq 0,p'(x)>0$, and $p''(x)>0$.

Thus, $p(x)$ is increasing and convex on $\left( 1,\infty  \right) $. By \eqref{diferenciasderoots}, we have $p(r)-p(q)=p(s)-p(r)$. Thus, since $s>r>q>1$, we conclude $r-q>s-r$, as desired.
\end{proof}

\begin{theorem}\label{denseness}
	For any $L\geq 2$, let $R_{L}$ be the set of roots of all incomplete PLRS's generated by $L$ coefficients. Then, for any $\varepsilon >0$, there exists an $M$ such that for all  $L>M$ and for any $\varepsilon $-ball $B_{\varepsilon }\subset \left( 1,2 \right) $, $B_{\varepsilon }\cap R_{L}\neq \varnothing$.
\end{theorem}

\begin{proof}
	Let $\varepsilon >0$ be arbitrary. By Lemma~\ref{rootsgotozero}, we may fix an $M$ such that for all $L>M$, $1<\lambda _{L}<1+\varepsilon $.

	From our previous work, we know that the sequence of length $L$ that has coefficients $[1,0,\ldots , 0,\left\lceil L\left( L+1 \right) /4 \right\rceil +1]$ is incomplete, as is any sequence of the form $[1,0,\ldots , 0,k]$, with $k \geq \left\lceil L\left( L+1 \right) /4 \right\rceil +1$.

	Note that $\lambda _{L}$ is the root of $[1,0,\ldots , 0,\left\lceil L\left( L+1 \right) /4 \right\rceil +1]$. Since $\lambda _{L}<1+\varepsilon $, it is clear that the root $\alpha $ of $[1,0,\ldots , 0,\left\lceil L\left( L+1 \right) /4 \right\rceil]$ fulfills  $1<\alpha <\lambda _{L}$, and so $\lambda -\alpha <\varepsilon $.

	Now, we know the sequence $[1,0,\ldots , 0,2^{L-1}]$ has a root of size exactly 2. Applying Lemma~\ref{rootsorder} iteratively, any two sequences $[1,0,\ldots , 0,k]$, $[1,0,\ldots , 0,k+1]$ with  $k\geq \left\lceil L\left( L+1 \right) /4 \right\rceil $ and roots $q,r$ must fulfill $r-q<\lambda _{L}-\alpha <\varepsilon $. Thus, any two consecutive sequences $[1,0,\ldots , 0,k]$, $[1,0,\ldots , 0,k+1]$ with $k\geq \left\lceil L\left( L+1 \right) /4 \right\rceil +1$ have roots with separation less than $\varepsilon $, and so the set of roots of sequences of the form $[1,0,\ldots , 0,k]$ with $\left\lceil L\left( L+1 \right) 4 \right\rceil +1\leq k\leq 2^{L-1}$ intercepts any $\varepsilon $-ball of $\left( 1,2 \right) $. As this is a subset of $R_{L}$, we are done. 
\end{proof}

\begin{corollary}\label{densenesscoro}
	The set of principal roots of incomplete sequences $R=\bigcup_{L=2}^{\infty }R_{L}$ is dense in $\left( 1,2 \right) $.
\end{corollary}

We conjecture that a similar result can be shown about complete roots; however, this proof has proven to be more difficult, as examples of families of complete sequences are more fragile.

\section{Open questions}
Here are conjectures and several other questions that future research could investigate.
\begin{itemize}
    \item Our results often focus on the final coefficient, such as in Theorems~\ref{decreaseLastCoe} and \ref{Adding M Theorem}. Do these results have any analogues for coefficients that are not the last?
    \item Can Theorem~\ref{thm:gbon} be extended to address what happens when $g < k$?
    \item Are there any other interesting families of PLRS's that can be fully characterized that have entries other than $0$ and $1$ as coefficients that are not the final coefficient?
    \item Are Conjectures~\ref{addFrontOnes} and \ref{2Lcrit} true?
    \item Is the missing component of the proof of Conjecture~\ref{final}, i.e., Conjecture~\ref{finalsegundamitad} true? 
\end{itemize}

\section{Acknowledgments} This research was conducted as part of the SMALL 2020 REU at Williams College. The authors were supported by NSF Grants DMS1947438 and DMS1561945, Williams College, Yale University, and the University of Rochester. The authors would like to thank the organizers of the 19th International Fibonacci Conference, the 2020 Young Mathematicians Conference, and CANT 2021 for the opportunity to present this work and receive feedback in earlier stages. 
\appendix

\section{Brown's criterion and a corollary}\label{apx:brownAndCor}
Here are several proofs of important results for our paper. All results will be restated for the reader's convenience.

\begin{theorem}{(Brown \cite{Br})}\label{thm:brownsCritApx}
	If $a_n$ is a non-decreasing sequence, then $a_n$ is complete if and only if $a_1 = 1$ and for all $n > 1$,
	\begin{equation}\label{eqn:brownsCritApx}
	a_{n+1} \leq 1+ \sum_{i = 1}^{n} a_i.
	\end{equation}
\end{theorem}

\begin{proof}
	Let $\left(a_n\right)_{n = 1}^{\infty}$ be a sequence of positive integers, not necessarily distinct, such that $a_1 = 1$ and 
	\begin{equation}
	    a_{n + 1} \leq 1 + \sum_{i = 1}^{n}{a_i} 
	\end{equation}
	for $n \geq 1$. Then, for $0<n<1+ \sum_{i = 1}^{k}{a_i}$, there exists $\left(b_i\right)_{i = 1}^{k}$, $b_i \in \{0, 1\}$ such that $n = \sum_{i = 1}^{k}{b_i a_i}$. We proceed by induction on $k$. The claim trivially holds for $k = 1$, so one may assume that the claim holds for $k = N$ as well. Hence, we must show that $0 < n < 1 +   \sum_{i = 1}^{N + 1}{a_i}$ implies the existence of $\left(\varepsilon_i\right)^{N + 1}_{i =1}$, $\varepsilon_i \in \{0, 1\}$ such that $n = \sum_{i = 1}^{N + 1}{\varepsilon_i a_i}$. Due to the inductive hypothesis, we only consider values satisfying
	\begin{equation}
	    1 + \sum_{i = 1}^{N}{a_i} \leq n < 1 + \sum_{i = 1}^{N + 1}{a_i}.
	\end{equation}
	Note that by assumption,
	\begin{equation}
	    n-a_{N + 1} \geq 1 + \sum_{i = 1}^{N}{a_i - a_{N + 1}} \geq 0. 
	\end{equation}
    Now, if $n - a_{N + 1} = 0$, the conclusion follows. Otherwise,
    \begin{equation}
        0 < n - a_{N + 1} < 1 + \sum_{i = 1}^{N}{a_i}
    \end{equation}
    implies the existence of $\left(b_i\right)_{i = 1}^{N}$ such that $n - a_{N + 1} = \sum_{i = 1}^{N}{b_i a_i}$. Then the result is immediate on transposing $a_{N + 1}$ and identifying $\varepsilon_i = b_i$ for $i \in \{1, \ldots, N\}$ and $\varepsilon_{N + 1} = 1$. This completes the sufficiency part of the proof.
    
    For the necessity, assume that there exists $n_0 \geq 1$ such that $a_{n_0 + 1} \geq 1 + \sum_{i = 1}^{n_0} a_i$. Then, however, \begin{equation}
        a_{n_0 + 1} > a_{n_0 + 1} - 1 > \sum_{i = 1}^{n_0} a_i,
    \end{equation}
    which implies that the positive integer $a_{n_0 + 1} - 1$ cannot be represented in the form $\sum_{i = 1}^{k}{b_i a_i}$. This leads to a contradiction and completes the proof.
\end{proof}

\begin{corollary}\label{cor:doublingCritApx}
	If $a_n$ is a nondecreasing sequence such that $a_1 = 1$ and $a_n \leq 2 a_{n - 1}$ for all $n \geq 2$, then $a_n$ is complete.
\end{corollary}

\begin{proof}
	We argue by induction on $n$ that $a_n$ satisfies Brown's criterion when $n \geq 2$. As $a_1 = 1$, for the base case we have
	\begin{equation}
		a_2 \leq 2 a_1 = 2 = a_1 + 1.
	\end{equation}
	Now assume for inductive hypothesis that for some $n \geq 2$,
	\begin{equation}
		a_n \leq a_{n - 1} + \cdots + a_1 + 1.
	\end{equation}
	Then
	\begin{equation}
		a_{n + 1} \leq 2 a_n = a_n + a_n \leq a_n + a_{n - 1} + \cdots + a_1 + 1,
	\end{equation}
	completing the induction.
\end{proof}

\begin{example}
	The converse does not hold. A sequence may be complete and have some terms that are larger than the double of the previous term. One such example is the sequence generated by $[1, 0, 1, 4]$, whose terms are $\left(1, 2, 3, 5, 11,\dots\right)$. Here, $11$ is more than twice $5$, yet the sequence is still complete.
\end{example}

\section{Lemmas for Section 2}\label{ProofsOfLemmas2}

\begin{lemma}\label{SequencesDifferencesAppendix}
Let $\left( G_{n} \right)$, $\left( H_{n} \right)$ be the sequences defined by $[c_1,\dots, c_{L}], \; [c_1,,\dots, c_{L},c_{L+1}]$, respectively, where $c_{L+1}$ is any positive integer. For all $k \geq 2$,
\begin{equation}
	H_{L+k}-G_{L+k}\geq 2\left( H_{L+k-1}-G_{L+k-1} \right). 
\end{equation} 
\end{lemma}
\begin{proof}
We use strong induction. 

We begin with the base case. First, recall that for all $n$ such that $1\leq n \leq L$, we know $H_{n}=G_{n}$. Further, note that
\begin{equation}
H_{L+1}=c_1H_{L}+\dots+c_{L}H_1+1=c_1G_{L}+\dots+c_{L}G_1+1=G_{L+1}+1.
\end{equation} 

Using this fact, we compute 
\begin{align}
	H_{L+2} &=c_1H_{L+1}+c_2H_{L}+\dots+c_{L}H_2+c_{L+1} H_1 \nnend
	&=c_1\left( G_{L+1}+1 \right)+c_2G_{L}+\dots+c_{L}G_2+c_{L+1} \nnend
	&=G_{L+1}+c_1+c_{L+1}.
\end{align}
Thus, we have that
\begin{equation}
	H_{L+2}-G_{L+2}=c_1+c_{L+1}\geq 2=2(1)= 2\left( H_{L+1}-G_{L+1} \right).
\end{equation} 
For the inductive step, suppose for some $m$, the lemma holds for all $2\leq k\leq m-1$. We wish to show the lemma holds for $m$, i.e., \begin{equation}
	H_{L+m}-G_{L+m}\geq 2\left( H_{L+m-1}-G_{L+m-1} \right).
\end{equation}
Expanding the terms using the recurrence definition, we see
\begin{equation}
 H_{L+m}-G_{L+m}\geq 2\left( H_{L+m-1}-G_{L+m-1} \right),
\end{equation}
which holds if and only if
\begin{equation}
\sum_{i=1}^{L}c_{i}H_{L+m-i}-\sum_{i=1}^{L}c_{i}G_{L+m-i}\geq 2\left( \sum_{i=1}^{L}c_{i}H_{L+m-1-i}-\sum_{i=1}^{L}c_{i}G_{L+m-1-i} \right).
\end{equation}
Note that for all $i \geq m$, $H_{L+m-i}-G_{L+m-i}=0$. We cancel out any such terms on both sides of the inequality above, simplifying to
\begin{equation}\label{InductHyp}
  \sum_{i=1}^{\min(m-1, L)}c_{i}\left( H_{L+m-i}-G_{L+m-i} \right) \geq  \sum_{i=1}^{\min(m-1, L)}  2c_{i}\left(H_{L+m-1-i}-G_{L+m-1-i}  \right).
\end{equation}
Note that for $m-1 \leq L$, we preserve the term $2c_{m-1}\left( H_{L}-G_{L} \right)=0$ in the right hand side sum, so that both sides of the inequality have the same number of summands.

By our inductive hypothesis, we see that for all $i$, 
\begin{equation}
c_{i}(H_{L+m-i}-G_{L+m-i})\geq 2c_{i}\left( H_{L+m-1-i}-G_{L+m-1-i} \right).
\end{equation}
Thus, inequality \eqref{InductHyp} holds, which completes the proof.
\end{proof}

\begin{lemma}\label{Lemma3.37Appendix}
Consider sequences $\left(G_n\right)=[c_1,c_2,\ldots, c_L]$ and   $\left(H_n\right)=[c_1,c_2,\ldots, k_L]$, where $1\leq k_L \leq c_L$. For all $k\in\N$,
\begin{equation}
H_{L+k+1}-2H_{L+k}\leq G_{L+k+1}-2G_{L+k}.
\end{equation}
\end{lemma}

\begin{proof}
We proceed by strong induction on $k$. For $k=1$, we have 
\begin{align}
     H_{L+2}-2H_{L+1}&=\left(c_1 H_{L+1}+c_2 H_{L}+\dots+k_L H_2 \right) - 2\left(c_1 H_{L}+c_2 H_{L-1}+\dots+k_L H_1\right)\nnend
     &=\left(c_1 H_{L+1}+c_2 G_{L}+\dots+k_L G_2 \right) - 2\left(c_1 G_{L}+c_2 G_{L-1}+\dots+k_L G_1\right) \nnend
     &=G_{L+2}-\left(G_{L+1}-H_{L+1}\right)-\left(2c_L-2k_L\right)-2G_{L+1}-2\left(c_L-k_L\right) \nnend
     &\leq  G_{L+2}-2G_{L+1}.
\end{align}
Assume the statement holds true for a natural number $k$. Now, note
\begin{align}
     H_{L+k+2}&-2H_{L+k+1}\nnend
     &=\left(c_1 H_{L+k+1}+c_2 H_{L+k}+\dots+k_L H_{k+2} \right) - 2\left(c_1 H_{L+k}+c_2 H_{L+k-1}+\dots+k_L H_{k+1}\right)\nnend
     &= c_1\left(H_{L+k+1}-2H_{L+k}\right)+c_2\left(H_{L+k}-2H_{L+k-1}\right)+\dots +k_L\left(H_{k+2}-2h_{k+1}\right)\nnend
     &\leq c_1\left(H_{L+k+1}-2H_{L+k}\right)+c_2\left(H_{L+k}-2H_{L+k-1}\right)+\dots +c_L\left(H_{k+2}-2H_{k+1}\right).
     \intertext{By the inductive hypothesis,} 
     &\leq c_1\left(G_{L+k+1}-2G_{L+k}\right)+c_2\left(G_{L+k}-2G_{L+k-1}\right)+\dots +c_L\left(G_{k+2}-2G_{k+1}\right) \nnend
     &=G_{L+k+2}-2G_{L+k+1}.
\end{align}
Therefore, the statement holds by induction.
\end{proof}

\begin{lemma}\label{Add 1 Lemma Appendix}
	Let $\left( G_{n} \right)$ be the sequence defined by $[c_1,\ldots, c_{L}]$, and let $\left( H_{n} \right)$ be the sequence defined by $[c_1,\ldots, c_{L-1}+1,\; c_{L}-1]$. Then, for all $k \geq 0$,
	\begin{equation}
		H_{L+k+1}-G_{L+k+1}\geq 2\left( H_{L+k}-G_{L+k} \right) .
	\end{equation} 	
\end{lemma}

\begin{proof}
We use strong induction. We begin with the base case. First, since the first $L-2$ coefficients of $\left( G_{n} \right),\left( H_{n} \right)$ are equivalent, we have that for all $1\leq n\leq L-1$, $G_{n}=H_{n}$. We also see that
\begin{equation}
	H_{L}=c_1H_{L-1}+\cdots+\left( c_{L-1}+1 \right) H_{1}+1=c_1G_{L-1}+\cdots+\left( c_{L-1}+1 \right) G_1+1=G_{L}+G_1=G_{L}+1.
\end{equation} 
Moreover,
\begin{equation}
\begin{array}{l@{{} {}}l}
H_{L+1}
&= c_1H_{L}+\cdots+\left( c_{L-1}+1 \right) H_2+\left( c_{L}-1 \right) H_1 \\
&= c_1\left( G_{L}+1 \right) +\cdots+\left( c_{L-1}+1 \right) G_2+\left( c_{L}-1 \right) G_1 \\
 &= c_1+G_2-G_1+\sum_{i=1}^{L}c_{i}G_{L+1-i} \\
 &= c_1+c_1 +G_{L+1}  \\
  &= 2c_1+G_{L+1}. \\
\end{array}
\end{equation}
Thus, we see that
\begin{equation}
	H_{L+1}-G_{L+1}=2c_1\geq 2=2(1)=2\left( H_{L}-G_{L} \right),
\end{equation} 
and so the base case holds.

For the induction step, suppose our lemma holds for all $0\leq k \leq m$. We wish to show this holds for $m+1$, so that $H_{L+m+1}-G_{L+m+1}\geq 2\left( H_{L+m}-G_{L+m} \right) $.

Since $\left(G_n\right)$ and $\left(H_n\right)$ are PLRS, we expand the terms in question using their respective recurrence relations to see that $H_{L+m+1}-G_{L+m+1}\geq 2\left( H_{L+m}-G_{L+m} \right)$ if and only if
\begin{multline}\label{InductIneq}
 \sum_{i=1}^{L}c_{i}H_{L+m+1-i}+H_{m+2}-H_{m+1}-\sum_{i=1}^{L}c_{i}G_{L+m+1-i} \\
 \geq 2\left( \sum_{i=1}^{L}c_{i}H_{L+m-1}+H_{m+1}-H_{m}-\sum_{i=1}^{L}c_{i}G_{L+m-i} \right).
\end{multline}

We note that by the induction hypothesis, we have that for all $i$,
\begin{equation}
c_{i}\left( H_{L+m+1-i}-G_{L+m+1-i} \right) \geq 2c_{i}\left( H_{L+m-i}-G_{L+m-i} \right).
\end{equation}

Moreover, $H_{m+2}-H_{m+1}\geq H_{m+1}-H_{m}$, simply because we know that gaps in a PLRS grow. 
Combining these two statements, we have that inequality $\ref{InductIneq}$ holds, and so our inductive step is complete.
\end{proof}

\begin{lemma}\label{Last Case Adding M Appendix}
Let $\left( G_{n} \right)$ be the sequence defined by $[c_1,\ldots, c_{L-1},1]$, and let $\left( H_{n} \right)$ be the sequence defined by $[c_1,\ldots,  c_{L-1}+1]$. Then, for all $k \geq  1$,
\begin{equation}
	H_{L+k+1}-G_{L+k+1}\geq 2\left( H_{L+k}-G_{L+k} \right). 
\end{equation}
\end{lemma}

\begin{proof}
The proof is similar to that of Lemma~\ref{Add 1 Lemma Appendix} and so we repeat our use of strong induction.

We begin with the base case. First, since first $L-2$ coefficients of $\left( G_{n} \right),\left( H_{n} \right)$ are equivalent, we have that for all $1\leq n\leq L-1$, $G_{n}=H_{n}$. In fact, even more can be said: $G_{L}=H_{L}$, as
\begin{equation}
	H_{L}=c_1H_{L-1}+\cdots +\left( c_{L-1}+1 \right) H_1=c_1G_{L-1}+\cdots +\left( c_{L-1}+1 \right) G_1=(G_{L}-1)+G_1=G_{L}.
\end{equation}
Hence, 
\begin{multline}
	H_{L+1}=c_1H_{L}+\cdots +(c_{L-1}+1)H_2=c_1G_{L}+\cdots +\left( c_{L-1}+1 \right) G_2\\ =G_{L+1}-G_1+G_2= G_{L+1}-\left( 1 \right) +\left( c_1+1 \right) =  G_{L+1}+c_1,
\end{multline}
and so we see that 
\begin{equation}
	H_{L+1}-G_{L+1}=c_1>0=2\left( H_{L}-G_{L} \right) .
\end{equation}
For the induction step, suppose for some $m$ that our lemma holds for all $0\leq k\leq m$. We wish to show this holds  for $m+1$, so that $H_{L+m+1}-G_{L+m+1}\geq 2\left( H_{L+m}-G_{L+m} \right) $.

Since $\left(G_n\right)$ and $\left(H_n\right)$ are PLRS, we expand the terms in question using their respective recurrence relations. On this basis, we can claim that $H_{L+m+1}-G_{L+m+1}\geq 2\left( H_{L+m}-G_{L+m} \right) $ if and only if
\begin{multline}\label{InductIneq 2}
 \sum_{i=1}^{L-1}c_{i}H_{L+m+1-i}+H_{m+2}-\sum_{i=1}^{L-1}c_{i}G_{L+m+1-i}-G_{m+1} \\
 \geq 2\left( \sum_{i=1}^{L-1}c_{i}H_{L+m-i}+H_{m}-\sum_{i=1}^{L}c_{i}G_{L+m-i} - G_{m-1} \right).
\end{multline}
By the induction hypothesis, we have that for all $i$, 
\begin{equation}
	c_{i}H_{L+m+1-i}-c_{i}G_{L+m+1-i}\geq 2\left( c_{i}H_{L+m-i}-G_{L+m-i} \right).
\end{equation} 
However, we can also show that $H_{m+2}-G_{m+1}\geq 2\left( H_{m+1}-G_{m} \right) $. By rewriting this as $H_{m+2}-2H_{m+1}\geq G_{m+1}-2G_{m}$, we see that for $m\leq L-1$, both sides are equal. For $m\geq L+1$, it suffices to note that $\left( H_{n} \right)$ grows faster, and thus so must the gaps between consecutive terms.
By combining these two observations, the inequality \eqref{InductIneq 2} holds, which completes the proof.
\end{proof}

\section{Lemmas for Section 3}\label{apx:sec3lemmas}

\begin{lemma}\label{lem:sharp1onekzero}
For the PLRS $H_{n+1} =H_n + NH_{n-k-1}$, with $N = \ceil*{(k + 2)(k + 3)/{4}}$, then 
\begin{equation}
(N-2)H_{n-k-1}\leq H_{n-1} +\dots+H_{n-k}.
\end{equation}
\end{lemma}

\begin{proof}
The claim follows by strong induction on $n$. Consider the base case, for $n=k+2$: $H_{n-k-1} = H_1 = 1, H_{n-k} = H_2 = 2, \dots, H_{n-1} = H_{k+1} = k+1$. 
\begin{align}
\left( N-2 \right) H_{n-k-1}  \leq  H_{n-1}+\dots+H_{n-k} &\iff  (N-2) \leq 2+3+\dots+k+\left( k+1 \right)  \nnend
	&\iff  \left\lfloor \frac{\left( k+2 \right) \left( k+3 \right) }{4}+\frac{1}{2}\right\rfloor  \leq \frac{\left( k+1 \right) \left( k+2 \right) }{2}+1 \nnend
	&\Longleftarrow \quad \frac{\left( k+2 \right) \left( k+3 \right) +2}{4}  \leq \frac{k^2+3k+2}{2}+1 \nnend
    &\iff  k^2+5k+8  \leq 2k^2+6k+8 \nnend
    &\iff  0  \leq k^2-k. 
\end{align}
Hence, the base case holds for $k\geq 0$.

For the induction hypothesis, assume the following holds for arbitrary, fixed $n$:
\begin{equation}
    (N-2)H_{n-k-1}\leq H_{n-1} +\dots+H_{n-k}.
\end{equation}
For the induction step, we wish to show the following:
\begin{equation}
    (N-2)H_{n-k}\leq H_{n} +\dots+H_{n-k+1}.
\end{equation}

Depending on the value of $n$, either the full recurrence relation is used, or an abridged version given by the definition of a PLRS is used. Thus, we consider two different cases for $n$.
First, for $k+2 < n < 2k+3$, the  recurrence relation implies that
\begin{equation}
    (N-2)H_{n-k}= (N-2)H_{n-k-1} + (N-2).
\end{equation}
By applying the induction hypothesis to the right hand side of the previous expression, we write
\begin{align}
    (N-2)H_{n-k} &\leq  H_{n-1} +\dots+H_{n-k} + N-2\nnend
    &= H_n + \cdots + H_{n-k+1} + (H_{n-k} - H_n + N-2).
\end{align}
As for all $k \geq 1$,
\begin{equation}
    H_{n}=H_{n-1} + NH_{n-k-2} \geq H_{n-k} + N - 2,
\end{equation}
then $(H_{n-k} - H_n + N-2) \leq 0$, and we conclude that
\begin{equation}
    (N-2)H_{n-k} \leq  H_n + \cdots + H_{n-k+1}.
\end{equation}

In the second case, where $n \geq 2k+3$, using the full recurrence relation, we write that
\begin{equation}
    (N-2)H_{n-k}= (N-2)H_{n-k-1} + N(N-2)H_{n-2k-2}.
\end{equation}
By applying the induction hypothesis to both terms on the right hand side of the previous expression, we write
\begin{align}
    (N-2)H_{n-k} &\leq  H_{n-1} +\dots+H_{n-k} + N(H_{n-k-2} +\dots+H_{n-2k-1})\nnend
    &= \sum_{i=1}^{k}\left(H_{n-i}+NH_{n-k-1-i}\right)\nnend
    &= \sum_{i=1}^k H_{n-i+1}.
\end{align}
Hence, the claim is true for all $n \geq k+1, k\geq 0$.
\end{proof}

\begin{lemma}\label{TrivId}
    $\sum_{i=1}^n 2^{i-1}i = 2^{n}(n-1) + 1$.
\end{lemma}

\begin{proof}
    By induction on $n$.
\end{proof}

\begin{lemma}\label{Line3Terms}
    Define $\left(f_n\right) = [\underbrace{1, \ldots, 1}_g]$ and $\left(H_n\right) = [\underbrace{1, \ldots, 1}_g, \underbrace{0, \ldots, 0}_k, 2^{k + 1} - 1]$. Then $H_{g + k + 1 + n} = f_{g + k + 1 + n} + (2^{k + 1} - 1)(2^n + 2^{n - 2}(n - 1))$ when $1 \leq n \leq g - k$.
\end{lemma}

\begin{proof}
    Define $a(n)$ so that $H_{g + k + 1 + n} = f_{g + k + 1 + n} + a(n)$ for $1 \leq n \leq g - k$. Now,
    \begin{align}
        H_{g + k + 1 + n} &= H_{g + k + n} + \cdots + H_{k + 1 + n} + (2^{k + 1} - 1)H_n\nnend
        &= \sum_{i = k + 1 + n}^{g + k + n} f_i + \sum_{i = 1}^{n - 1} a(i) + \sum_{i = 1}^{k + 1} 2^{i - 1} + (2^{k + 1} - 1)2^{n - 1}
        \intertext{(since $k + 1 + n \leq g + 1$, $(H_i - f_i)$ spans all the indices from $k + 1 + n \leq g + 1$ to $g + k + n$)}
        &= f_{g + k + n + 1} + \sum_{i = 1}^{n - 1} a(i) + (2^{k + 1} - 1)(2^{n - 1} + 1)
    \end{align}
    Therefore,
    \begin{equation}
        a(n) = \sum_{i = 1}^{n - 1} a(i) + (2^{k + 1} - 1)(2^{n - 1} + 1)
    \end{equation}
    and
    \begin{equation}
        a(n - 1) = \sum_{i = 1}^{n - 2} a(i) + (2^{k + 1} - 1)(2^{n - 2} + 1).
    \end{equation}
    Hence
    \begin{equation}
        a(n) = 2a(n - 1) + (2^{k + 1} - 1)(2^{n - 2}).
    \end{equation}
    Since $a(1) = 2(2^{k + 1} - 1)$, by induction we have
    \begin{equation}
        a(n) = (2^{k + 1} - 1)(2^n + 2^{n - 2}(n - 1)).
    \end{equation}
\end{proof}

\begin{lemma}\label{Line4Terms}
    Define $\left(f_n\right) = [\underbrace{1, \ldots, 1}_g]$ and $\left(H_n\right) = [\underbrace{1, \ldots, 1}_g, \underbrace{0, \ldots, 0}_k, 2^{k + 1} - 1$. Then $f_{g + n} = 2^{g + n - 1} - 2^{n - 2}(n + 1)$ when $1 \leq n \leq g$.
\end{lemma}

\begin{proof}
    Set $f_{g + n} = 2^{g + n - 1} - a(n)$ for $1 \leq n \leq g$. Then
    \begin{align}
        f_{g + n} &= f_{g + n - 1} + \cdots + f_n\nnend
        &= 2^{g + n - 2} + 2^{g + n - 3} + \cdots + 2^{n - 1} - (a(n - 1) + \cdots + a(1))\nnend
        &= 2^{g + n - 1} - \left( 2^{n - 1} + \sum_{i = 1}^{n - 1} a(i) \right).
    \end{align}
    Therefore,
    \begin{equation}
        a(n) = 2^{n - 1} + \sum_{i = 1}^{n - 1} a(i)
    \end{equation}
    and
    \begin{equation}
        a(n - 1) = 2^{n - 2} + \sum_{i = 1}^{n - 2} a(i).
    \end{equation}
    Hence
    \begin{equation}
        a(n) = 2^{n - 1} + 2a(n - 1) - 2^{n - 2} = 2a(n - 1) + 2^{n - 2}.
    \end{equation}
    Since $a(1) = 1$, by induction, we have $a(n) = 2^{n - 2}(n + 1)$.
\end{proof}

\begin{lemma}\label{lem:Gap1}
    For $k + \lceil \log_2k \rceil \leq g < 2k$, $\left(H_n\right)$ defined as $[\underbrace{1, \dots, 1}_{g}, \underbrace{0, \dots, 0}_{k}, 2^{k+1} - 1]$, we have 
    \begin{equation}
    2^{g+k+1} - \sum_{i=k+n+2}^{g+k+1}H_i \leq 2^g + 2^{k+n+2} - 2^{n+1}
    \end{equation}
    for all $g - k \leq n \leq k$.
\end{lemma}

\begin{proof}
We proceed by induction on $n$. Suppose it holds for some $n \geq g-k$. Then 
\begin{equation}
    2^{g+k+1} - \sum_{i=k+n+3}^{g+k+1}H_i = 2^{g+k+1} - \sum_{i=k+n+2}^{g+k+1}H_i + H_{k+n+2}.
\end{equation}
By the induction hypothesis,
\begin{equation}
    2^{g+k+1} - \sum_{i=k+n+3}^{g+k+1}H_i \leq 2^g + 2^{k+n+2} - 2^{n+1} + H_{k+n+2}.
\end{equation}
As we can check explicitly that $H_{k+n+2} \leq 2^{k+n+1} \leq 2^{k+n+2} - 2^{n+1}$, we see
\begin{align}
    2^{g+k+1} - \sum_{i=k+n+3}^{g+k+1}H_i &\leq 2^g + 2^{k+n+2} - 2^{n+1} + (2^{k+n+2} - 2^{n+1}) \nnend
    &= 2^g + 2^{k+n+3} - 2^{n+2}.
\end{align}
It remains to show for the base case $n = g-k$. This can be shown directly from the given formulas in Theorem~\ref{thm:SharpgAndlog2k}.
\end{proof}

\begin{lemma}\label{lem:Gap2}
    For $\left(H_n\right)$ defined as in Lemma~\ref{lem:Gap1} and the same conditions on $g$ and $k$, we have
    \begin{equation}
    H_{(g+k+1) + n} \geq (2^{k+1} - 1)(2^{g+n-1} - 2^{n-2}(n+1))
    \end{equation}
    for all $1 \leq n \leq k$.
\end{lemma}

\begin{proof}
    This is equivalent to showing that
    \begin{equation}
    2^{g+k+n} - H_{(g+k+1) + n} \leq 2^{g+n-1} + 2^{k+n-1}(n+1) - 2^{n-2}(n+1) \ \text{ for all } 1 \leq n \leq k,
    \end{equation}
    which we proceed to prove by strong induction on $n$. The case $1 \leq n \leq g-k$ has been established in Theorem~\ref{thm:SharpgAndlog2k}, so we suppose this holds for all $n \leq m$ for some $g-k \leq m < k$. Then
    \begin{align}
    &2^{g+k+(m+1)} - H_{(g+k+1) + (m+1)} \nnend
    &\quad= 2^{g+k+m+1} - \left(\sum_{i=1}^m H_{(g+k+1)+i} + \sum_{i=k+m+2}^{g+k+1}H_i + (2^{k+1}-1)H_{m+1} \right) \nnend
    &\quad= \sum_{i=1}^m \left(2^{g+k+i} - H_{(g+k+1)+i} \right) + \left(2^{g+k+1} - \sum_{i=k+m+2}^{g+k+1}H_i \right) - (2^{k+1}-1)2^m.
    \intertext{By the inductive hypothesis and Lemma~\ref{lem:Gap1},}
    &\quad\leq \sum_{i=1}^m \left(2^{g+i-1} + 2^{k+i-1}(i+1) - 2^{i-2}(i+1) \right) + \left( 2^g + 2^{k+m+2} - 2^{m+1} \right) - (2^{k+1}-1)2^m \nnend
    &\quad= 2^{g+(m+1) - 1} + 2^{k + (m+1) - 1}((m+1)+1) - 2^{(m+1) - 2} ((m+1)+1).
    \end{align}
    Our proof by induction is complete.
\end{proof}

\begin{lemma}\label{lem:lastCoeff3}
The sequence generated by $[1,\dots,1,0,3]$, with $k \geq 1$ ones, is always complete.
\end{lemma}

\begin{proof}
By strong induction on $n$.

A simple calculation can be done to show that Brown's criterion holds for all $m \leq k+1$, i.e. $H_{m+1} \leq 1 + H_1 + \dots + H_{m}$. 

For the induction hypothesis, assume that for some $n \geq k+1$, Brown's criterion holds for all $m<n$, i.e., assume $H_{m+1} \leq 1 + H_1 + \dots +H_{m}$ for all $m<n$.

For the induction step, we start with the recurrence relation and apply the induction hypothesis:
\begin{align}
H_{n+1} &= H_n + \dots + H_{n-k+1} + 3H_{n-k-1}\nnend
&\leq H_n + \dots + H_{n-k+1} + H_{n-k} + 2H_{n-k-1}\nnend
&\leq H_n + \dots + H_{n-k+1} + H_{n-k} + H_{n-k-1} + H_{n-k-2} + \dots + H_1 +1.
\end{align}
Hence, by Brown's criterion, the sequence is complete. By strong induction, the lemma is proved.
\end{proof}

\begin{lemma}\label{firstL+2append}
	Let $\left( H_{n} \right)$ defined by $[1,0,\dots, 0, \underbrace{1,\dots, 1}_{m},N]$ be a PLRS with $L$ coefficients. Then, if the sequence is incomplete, it must fail Brown's criterion at the $L+1$\textsuperscript{th} or $L+2$\textsuperscript{th} term. In other words, if $H_{L+1}\leq 1+\sum_{i=1}^{L}H_{i}$ and $H_{L+2}\leq 1+\sum_{i=1}^{L+1}H_{i}$, then $\left( H_{n} \right)$ is complete.
\end{lemma}

\begin{proof}
Let $\left( H_{n} \right)$ be defined as above; it is clear that the first $L$ terms pass Brown's criterion. Now suppose the sequence passes Brown's criterion at the $L+1$-ist and $L+2$-nd term, so that
\begin{equation}
H_{L+1}\leq \sum_{i=1}^{L}H_{i}+1, \ \ \ H_{L+2}\leq \sum_{i=1}^{L+1}H_{i}+1.
\end{equation} 
We show that $\left( H_{n} \right)$ is complete.

We show by induction that if $\left( H_{n} \right)$ satisfies Brown's criterion at the $L+2$nd term, then it satisfies Brown's criterion at the $L+k$\textsuperscript{th} term, for any $2\leq k\leq L-1$. We assume our base case of $k=2$ by hypothesis, so only the induction step remains to be shown. 

Suppose for some $k$ that
\begin{equation}\label{a}
	H_{L+k}=H_{L+k-1}+H_{k+m}+\dots+H_{k+1}+NH_{k} \leq \sum_{i=1}^{L+k-1}H_{i}+1.
\end{equation}

We wish to show that 
\begin{equation}\label{b}
H_{L+k+1}=H_{L+k}+H_{k+m+1}+\dots+H_{k+2}+NH_{k+1}\leq \sum_{i=1}^{L+k}H_{i}+1.
\end{equation}
Looking at the difference between equations \eqref{a} and \eqref{b}, we see it suffices to show that 
\begin{equation}
	\left( H_{L+k}-H_{L+k-1} \right) +N\left( H_{k+1}-H_{k} \right) +H_{k+m+1}-H_{k+1}\leq H_{L+k}.
\end{equation} 
Or equivalently, 
\begin{equation}\label{c}
	N\left( H_{k+1}-H_{k} \right) +H_{k+m+1}-H_{k+1}\leq H_{L+k-1}.
\end{equation}
Expanding $H_{L+k-1}$, 
\begin{equation}\label{f}
H_{L+k-1} = H_{L+k-2}+H_{k+m-1}+\dots+H_{k}+NH_{k-1}.
\end{equation}
We can repeatedly expand the largest term of expression \eqref{f}, giving us longer partial sums. In particular, applying the process $k$ times, we see:
\begin{align}\label{d}
H_{L+k-1}
\ &=\ H_{L+k-2}+\sum_{i=k}^{k+m-1}H_{i}+NH_{k-1} \nnend
&=\ \left( H_{L+k-3}+\sum_{i=k-1}^{k+m-2}H_{i} +NH_{k-2} \right) +\sum_{i=k}^{k+m-1}H_{i} +NH_{k-1} \nnend
&=\ \left( H_{L+k-3}+\sum_{i=k-2}^{k+m-3}H_{i}+NH_{k-3} \right) +\sum_{i=k-1}^{k+m-2}H_{i}+\sum_{i=k}^{k+m-1}H_{i}  +N(H_{k-1}+H_{k-2}) \nnend
 &\ \;\vdots \nnend
 &=\ H_{L-1}+ \sum_{i=1}^{m}H_{i}+\cdots +\sum_{i=k}^{k+m-1}H_{i}+ N\left( H_{k-1}+H_{k-2}+\dots+H_1 \right)  \nnend
 &=\ H_{L-1}+\sum_{a=1}^{k}\sum_{i=a}^{a+m-1}H_{i}+N\left( \sum_{i=1}^{k-1}H_{i} \right).
\end{align}
Thus inequality \eqref{c} becomes
\begin{equation}\label{e}
	N\left( H_{k+1}-H_{k}-H_{k-1}-\dots-H_2-H_1 \right)+H_{k+m+1} \leq H_{L-1}+H_{k+1}+\sum_{a=1}^{k}\sum_{i=a}^{a+m-1}H_{i}.
\end{equation}
Assuming $k<L$, we can write $H_{k+1}=H_{k}+H_{k-L+m+1}+\cdots +H_{k-L+2}$, where for the sake of notation we define $H_0=1$ and $H_{j}=0$ for all $j< 0$. Thus,
\begin{align} 
N\left( H_{k+1}-H_{k}-H_{k-1}-\dots-H_2-H_1 \right) &= -N\left( H_{k-1}+H_{k-2}+\cdots+H_{k-L+m+2} \right) \nnend
&<-NH_{k-1},
\end{align} 
and so for inequality \eqref{e} it suffices to show
\begin{equation}
	H_{k+m+1}\leq H_{L-1}+H_{k+1}+\sum_{a=1}^{k}\sum_{i=a}^{a+m-1}H_{i}+ NH_{k-1}.
\end{equation}
Expanding the left hand side, we wish to show 
\begin{align}
H_{k+m+1}
&= H_{k+m}+H_{k-L+2m+1}+\cdots+H_{k-L+1}+NH_{k-L+m} \nnend
&\leq    H_{L-1}+H_{k+1}+\sum_{a=1}^{k}\sum_{i=a}^{a+m-1}H_{i}+ NH_{k-1}.
\end{align}
As $k-L+m \leq k-3<k-1$, we see $NH_{k-L+m}<NH_{k-1}$, and so we need only show
\begin{equation}\label{x}
	H_{k+m}+H_{k-L+2m+1}+\cdots+H_{k-L+1}\leq  H_{L-1}+H_{k+1}+\sum_{a=1}^{k}\sum_{i=a}^{a+m-1}H_{i}. 
\end{equation}
As $k\geq 2$, we note that in the double sum $\sum_{a=1}^{k}\sum_{i=a}^{a+m-1}H_{i}$, the summands $H_1,H_{k+m-1}$ are present exactly once, and for any $1<i<k+m-1$, the summand $H_{i}$ is present at least twice. Thus we can take the crude bound
\begin{equation}
	\sum_{a=1}^{k}\sum_{i=a}^{a+m-1}H_{i}\geq \sum_{i=1}^{k+m-1}H_{i}+\sum_{i=2}^{k+m-2}H_{i}.
\end{equation} 
Applying this bound on the right hand side of inequality \eqref{x}, and taking the trivial bounds $H_{L-1}>H_1+1,\; H_{k+1}>1$, we see 
\begin{equation}
	H_{L-1}+H_{k+1}+\sum_{a=1}^{k}\sum_{i=a}^{a+m-1}H_{i} \geq \left(  \sum_{i=1}^{k+m-1}H_{i}+1 \right)  + \left( \sum_{i=1}^{k+m-2}H_{i}+1 \right).
\end{equation} 
As we assumed $\left( H_{n} \right)$ fulfills Brown's criterion for terms below $L+k$, we know
\begin{equation}
H_{k+m}\leq \sum_{i=1}^{k+m-1}H_{i}+1.
\end{equation} 
Finally, it is clear that 
\begin{equation}
H_{k-L+2m+1}+\cdots+H_{k-L+1} \leq \sum_{i=1}^{k+m-2}H_{i}+1,
\end{equation}
as no indices on the left sum are repeated. Combining these two facts, we have
\begin{align}
	H_{k+m}+H_{k-L+2m+1}+\cdots+H_{k-L+1}
    &\leq \left(  \sum_{i=1}^{k+m-1}H_{i}+1 \right)  + \left( \sum_{i=1}^{k+m-2}H_{i}+1 \right) \nnend
    &\leq H_{L-1}+H_{k+1}+\sum_{a=1}^{k}\sum_{i=a}^{a+m-1}H_{i}.
\end{align}
Thus inequality \eqref{x} holds, and we are done.
\end{proof}

\begin{lemma}\label{diffGH}
    Let $\left(G_n\right)$ and $\left(H_n\right)$ be PLRS's, both with $L$ coefficients, which are defined by $[1, 0, \dots, 0, \underbrace{1, \dots, 1}_{m}, N]$ and $[1, 0, \dots, 0, \underbrace{1, \dots, 1}_{m+1}, N+1]$ respectively. For $(L-1) / 2 \leq m \leq L-4$,
    \begin{equation} \begin{cases}
        H_{i-1} = G_i -1, & \text{if $2 \leq i< 2(L-m)$;}\\
        H_{i-1} = G_{i}, & \text{if $i=2(L-m)$;}\\
        H_{i-1} > G_i, & \text{if $2(L-m) < i \leq L$.}\\
    \end{cases}
    \end{equation}
\end{lemma}

\begin{proof}
From the proof of Lemma~\ref{sum2L-2m}, we get 
\begin{equation}
\begin{cases}
    H_n = n, H_{(L-m-1) + n} = L - m - 1 + n + \frac{n(n+1)(n+2)}{6}, & \text{if $1 \leq n \leq L-m-1$;}\\[1.5mm]
    G_n = n, G_{L-m + n} = L - m + n + \frac{n(n+1)(n+2)}{6}, & \text{if $1 \leq n \leq L-m$}.
\end{cases}
\end{equation}
From these explicit formulas, $H_{i-1} = G_i -1$ for all $2 \leq i \leq 2(L-m)-1$. Now, 
\begin{align}
H_{2(L-m)-1} 
\ &=\ H_{2(L-m-1)} + H_{L-m} + \sum_{i=1}^{L-m-1}H_i + 1 \nnend
&=\ 2(L-m-1) + \frac{(L-m-1)(L-m)(L-m+1)}{6} + 1 + \sum_{i=1}^{L-m-i}G_i + 1 \nnend
&=\ 2(L-m)-1 + \frac{(L-m-1)(L-m)(L-m+1)}{6} + \sum_{i=1}^{L-m-i}G_i + 1 \nnend
&=\ G_{2(L-m)-1} + \sum_{i=1}^{L-m-i}G_i + 1 \nnend
&=\ G_{2(L-m)}.
\end{align}
Similarly, by writing out explicit formulas, one can show that $H_{2(L-m)} > G_{2(L-m) + 1}$. Also, it is clear that $H_i \geq G_i$ for all $i$. Therefore, for any $2(L-m)+1 < k \leq L$,
\begin{align}
    H_{k-1} - G_k &= (H_{k-2} - G_{k-1}) + \sum_{i=1}^{k-(L-m)}(H_i - G_i)\nnend
    & \geq H_{k-2} - G_{k-1},
\end{align}
and the last inequality follows by induction on $k$.
\end{proof}

\begin{lemma}\label{lem:trivialineq}
    If $m + 3 < 2(L-m)$ and $m \geq (L-1) / 2$, then
    \begin{equation}
    2m-L + \sum_{i = L-m+2}^{m+1}i \geq 2L - 3(m+1) + 2.
    \end{equation}
\end{lemma}

\begin{proof}
This is equivalent to 
\begin{equation}
\frac{(m+1)(m+2)}{2} - \frac{(L-m+1)(L-m+2)}{2} \geq 3L - 5m - 1,
\end{equation}
which simplifies to 
\begin{equation}
L(2m-L) + 16m + 2 \geq 8L,
\end{equation}
which is true since $L(2m-L) > 6$ and $2m + 1 \geq L$.
\end{proof}

\section{Lemmas for Section 4}\label{apx:sec4lemmas}

\begin{lemma}\label{appendixitAppendix}
For any $L \in \Z_{>0}$, let $\lambda _{L}$ be the principal root of $x^{L} - x^{L - 1} - N_L - 1$. Then
\begin{equation}\label{eq:LowerBoundPrincipalRoot}
	 \lambda _{L}-1  \geq \frac{L+2}{L^2+L+4}.
\end{equation}
\end{lemma}

\begin{proof}
Set $f(x) = x^L - x^{L - 1} - N_L - 1$. It suffices to show that $f((L + 2)/(L^2 + L + 4) + 1) \leq 0$. For $L = 2$, $N_L = \lceil 2(2 + 1)/4 \rceil = 2$, so $f(x) = x^2 - x - 3$. Thus,
\begin{equation}
    f\left(\frac{L + 2}{L^2 + L + 4} + 1\right) = -\frac{61}{25} < 0,
\end{equation}
and we now consider $L \neq 2$. Since $1/(L^2 + L + 4) < 1/(L^2 - 4)$, it suffices to show
\begin{equation}
    \lambda_L - 1 \geq \frac{L + 2}{L^2 - 4} = \frac{1}{L - 2} \quad\iff\quad \lambda_L \geq \frac{L - 1}{L - 2}.
\end{equation}
As noted above, it suffices to show $f((L - 1)/(L - 2)) \leq 0$. That is,
\begin{align}
    \left(\frac{L - 1}{L - 2}\right)^L - \left(\frac{L - 1}{L - 2}\right)^{L - 1} - \left\lceil\frac{L(L + 1)}{4}\right\rceil - 1 &\leq 0\\
    \left(\frac{L - 1}{L - 2}\right)^L - \left(\frac{L - 1}{L - 2}\right)^{L - 1}  &\leq \left\lceil\frac{L(L + 1)}{4}\right\rceil + 1.\label{eq:sufficientForPrincipalBound}
\end{align}
Equality holds for $L = 3$ and the left hand side is decreasing while the right hand side is increasing, so \eqref{eq:sufficientForPrincipalBound} holds for $L \geq 3$. When $L = 1$, the left hand side is negative while the right hand side is positive, so we have covered all $L \in \Z_{>0}$, completing the proof.
\end{proof}

\bigskip
\hrule
\bigskip

\noindent {\it 2010 Mathematics Subject Classification}: Primary 11B37, Secondary 11B39.

\noindent Keywords: Zeckendorf's theorem, Positive linear recurrence sequence, Complete sequence, Brown's criterion, Characteristic polynomial.

\bigskip
\hrule
\bigskip

\noindent (Concerned with sequences \seqnum{A000045}, \seqnum{A000079}, \seqnum{A006138}, \seqnum{A054925}.)

\bigskip
\hrule
\bigskip

\vspace*{+.1in}
\noindent

\bigskip
\hrule
\bigskip

\noindent
Return to
\htmladdnormallink{Journal of Integer Sequences home page}{https://cs.uwaterloo.ca/journals/JIS/}.
\vskip .1in

\end{document}